\documentclass[11pt,onecolumn,draftcls]{IEEEtran} 
\ifCLASSINFOpdf
\else
\fi
\usepackage[cmex10]{amsmath}
\usepackage{amssymb,amsthm}
\usepackage[utf8]{inputenc}
\usepackage[T1]{fontenc}
\usepackage{lmodern}          
\usepackage[ngerman,english]{babel}
\usepackage{pict2e}
\usepackage[pdftex]{graphicx}
\usepackage{subfigure}
\usepackage{listings}  
\usepackage{listliketab}
\usepackage{color}
\usepackage{verbatim} 
\usepackage{changes}
\usepackage{selinput} % Erlaubt durchstreichen

\usepackage{booktabs}

\newtheorem{theorem}{Theorem}
\newtheorem{lemma}[theorem]{Lemma}

\theoremstyle{definition}
\newtheorem{definition}[theorem]{Definition}

\theoremstyle{remark}

\def\N{{\mathbb N}} %N, Q, R und C als K"orperzeichen bold
\def\Z{{\mathbb Z}}

\def\R{{\mathbb R}}

\def\Norm{\Vert}

\def\skpl{\langle}      % Skalarprodukt links
\def\skpr{\rangle}      % Skalarprodukt rechts

\newcommand{\cO}{\mathcal{O}}
\newcommand{\cR}{\mathcal{R}}
\newcommand{\gr}{\mathrm{gr\,}}
\newcommand{\st}{\,:\,}
\newcommand{\be}{\begin{equation}}
\newcommand{\ee}{\end{equation}}

\newcommand{\inn}[2]{\langle #1,#2 \rangle}
\newcommand{\ml}{\vskip 4pt\noindent}

\begin{document}

%
% paper title
% can use linebreaks \\ within to get better formatting as desired
% Do not put math or special symbols in the title.
\title{Multigrid Convergence for the MDCA-Curvature Estimator}
%
%
% author names and IEEE memberships
% note positions of commas and nonbreaking spaces ( ~ ) LaTeX will not break
% a structure at a ~ so this keeps an author's name from being broken across
% two lines.
% use \thanks{} to gain access to the first footnote area
% a separate \thanks must be used for each paragraph as LaTeX2e's \thanks
% was not built to handle multiple paragraphs
%

\author{Andreas~Schindele,
        Peter~Massopust
        and~Brigitte~Forster,~\IEEEmembership{Member,~IEEE}% <-this % stops a space
\thanks{This work was partially supported by the DFG grants FO 792/2-1 and MA 5801/2-1 awarded to Brigitte Forster and to Peter {Massopust}.}
\thanks{A. Schindele is with the Faculty of Mathematics and Computer Science, Julius-Maximilians Universit\"at W\"urzburg, Germany (e-mail: andreas.schindele@mathematik.uni-wuerzburg.de).}
\thanks{B. Forster is with the Faculty of Computer Science and Mathematics, Universit\"at Passau, Germany (e-mail: brigitte.forster@uni-passau.de).}
\thanks{P. Massopust is with the Helmholtz Zentrum M\"unchen and with the Department of Mathematics, Technische Universit\"at M\"unchen, Germany (e-mail: massopust@ma.tum.de).}
%\thanks{M. Shell is with the Department
%of Electrical and Computer Engineering, Georgia Institute of Technology, Atlanta,
%GA, 30332 USA e-mail: (see http://www.michaelshell.org/contact.html).}% <-this % stops a space
%\thanks{J. Doe and J. Doe are with Anonymous University.}% <-this % stops a space
%\thanks{Manuscript received April 19, 2005; revised December 27, 2012.}
}

% note the % following the last \IEEEmembership and also \thanks - 
% these prevent an unwanted space from occurring between the last author name
% and the end of the author line. i.e., if you had this:
% 
% \author{....lastname \thanks{...} \thanks{...} }
%                     ^------------^------------^----Do not want these spaces!
%
% a space would be appended to the last name and could cause every name on that
% line to be shifted left slightly. This is one of those "LaTeX things". For
% instance, "\textbf{A} \textbf{B}" will typeset as "A B" not "AB". To get
% "AB" then you have to do: "\textbf{A}\textbf{B}"
% \thanks is no different in this regard, so shield the last } of each \thanks
% that ends a line with a % and do not let a space in before the next \thanks.
% Spaces after \IEEEmembership other than the last one are OK (and needed) as
% you are supposed to have spaces between the names. For what it is worth,
% this is a minor point as most people would not even notice if the said evil
% space somehow managed to creep in.

% The paper headers
\markboth{IEEE Transactions on Image Processing}%
{Schindele \MakeLowercase{\textit{et al.}}: Multigrid Convergence}
% The only time the second header will appear is for the odd numbered pages
% after the title page when using the twoside option.
% 
% *** Note that you probably will NOT want to include the author's ***
% *** name in the headers of peer review papers.                   ***
% You can use \ifCLASSOPTIONpeerreview for conditional compilation here if
% you desire.

% If you want to put a publisher's ID mark on the page you can do it like
% this:
%\IEEEpubid{0000--0000/00\$00.00~\copyright~2012 IEEE}
% Remember, if you use this you must call \IEEEpubidadjcol in the second
% column for its text to clear the IEEEpubid mark.

% use for special paper notices
%\IEEEspecialpapernotice{(Invited Paper)}

% make the title area
\maketitle

% As a general rule, do not put math, special symbols or citations
% in the abstract or keywords.
\begin{abstract}
We consider the problem of estimating the curvature profile along the boundaries of digital objects in segmented black-and-white images. We start with the curvature estimator proposed by Roussillon et al.,
%\cite{rouss}, 
which is based on the calculation of \emph{maximal digital circular arcs} (MDCA). We extend this estimator to the $\lambda$-MDCA curvature estimator that considers several MDCAs for each boundary pixel and is therefore smoother than the classical MDCA curvature estimator.

We prove an explicit order of convergence result 
%of $\mathcal{O}(h^{\frac13})$ 
for convex subsets in $\mathbb{R}^2$ with positive, continuous curvature profile. In addition, we evaluate the curvature estimator on various objects with known curvature profile. We show that the observed order of convergence is close to the theoretical limit of $\mathcal{O}(h^{\frac13})$. Furthermore, we establish that the $\lambda$-MDCA curvature estimator outperforms the MDCA curvature estimator, especially in the neighborhood of corners.
\end{abstract}

% Note that keywords are not normally used for peerreview papers.

% For peer review papers, you can put extra information on the cover
% page as needed:
% \ifCLASSOPTIONpeerreview
% \begin{center} \bfseries EDICS Category: 3-BBND \end{center}
% \fi
%
% For peerreview papers, this IEEEtran command inserts a page break and
% creates the second title. It will be ignored for other modes.
\IEEEpeerreviewmaketitle

\section{Curvature detection -- a task for image analysis}

Line, circle and more generally edge detection methods are well-known and well-understood tasks in image processing. Currently, more sophisticated applications aim at extracting valuable information from the detected edges, e.g., their curvature profile. The idea is to use this additional information for image analysis or interpretation purposes.

Young et al. \cite{nbe} use the curvature of a segmented shape to automatically derive the so-called bending energy for biological shape analysis. The bending energy aims at modeling a restricting parameter in developmental processes (e.g. blood cells \cite{canham}), or as van Vliet et al. \cite{nbe2} put it, to ``quantify the energy stored in the shape of the contour.''  Duncan et al. \cite{heart} use the bending energy to analyze and quantify the shape characteristics of parts of the endocardial contour. The concept of bending energy is also applied in the quantitative shape analysis of cancer cells \cite{cell, cell2}, leaf phenotyping \cite{backhaus}, lipid layers \cite{sodt}, and others. Applications of curvature profiles are found in face recognition \cite{face} or diatom identification \cite{kiesel} which is widely used to monitor water quality or environmental changes.  For all these applications, it is crucial to obtain a best possible curvature estimation from the given digitized data.  This, however, is a nontrivial task since through the discretization process information and resolution are lost. Therefore, upper and lower bounds for curvature estimates derived from digital data are needed.

To this end, we consider the novel curvature estimator introduced by Roussillon et al. \cite{rouss}. This estimator is based on the calculation of Maximal Digital Circular Arcs (MDCA) fitted to a curve yielding an estimate for the curvature via the inverse radius of the osculating circle. The performance of this technique is compared to other estimation techniques in \cite{multi}.  In \cite{rouss}, a first convergence result for multigrid estimation is proved and in \cite{multi} the order of convergence is given by $\Omega ( h^{\frac1 2})$ as the grid size $h \to 0$. However, for application purposes, some prerequisites are not directly verifiable, e.g., the behavior of the length of a covering ring segment and its relation to the grid size of the discretization. In this article, we reconsider the proof of the result in \cite{multi}, extend it to the sharper $\mathcal{O}(h^{\frac 1 3})$ convergence rate, and replace the original assumptions by more easily accessible and verifiable conditions. Moreover, we introduce an extension of the MDCA-estimator, namely the $\lambda$-MDCA estimator, which averages the results of the MDCA-procedure. This has the effect that in neighborhoods of critical points, such as edges, the precision of the estimation results is increased. Numerical simulations illustrate our convergence result and encourage further enquiry into $\lambda$-averaging.

The article is organized as follows: In Sections \ref{diffgeo} and \ref{sec imaging}, we briefly revisit the required notions from differential geometry and image processing. In Section \ref{sec Main Task}, we formulate the main task: the optimal estimation of the curvature profile. Section \ref{MDCAchap} introduces the MDCA und the $\lambda$-MDCA estimator. The proofs of the multigrid convergence rate and of the main results are presented in Section \ref{errest}. Experimental validations are offered in Section \ref{ergeb} and the Conclusions \ref{sec conclusions} close the article.

\section{Brief Review of Differential Geometry}\label{diffgeo}

In this section, we briefly introduce those basic concepts of differential geometry that are relevant for the remainder of this paper. For more details and proofs, the interested reader is referred to, for instance, \cite{docarmo,kuehnel,montiel}.

We consider twice continuously differentiable mappings $\gamma : [a,b]\to \R^2$, so-called planar Frenet curves. A Frenet curve $\gamma$ is called piecewise Frenet if it is a Frenet curve on each of its subdomains. The tangent vector to a curve $\gamma$ is defined by $\dot{\gamma} (t) := \frac{d\gamma}{dt}(t)$. The set $\gr \gamma:=\{\gamma(t)\st t\in[a,b]\}$ denotes the graph of $\gamma$. A curve $\gamma$ is regular if $\dot{\gamma}(t) \neq 0$, for all $t\in [a,b]$. For regular curves, we define the length of $\gamma$ to be the integral 
\[
L_{\gamma} := \int_a^b \Norm\dot{\gamma}(t)\Norm dt.
\]
For each $t\in [a,b]$, the arc length of a regular curve $\gamma$ is defined by $$s(t) := \int_a^t \Norm\dot{\gamma}(t)\Norm dt.$$ Every regular curve can be parametrized by arc length $s$. In this case, the norm of the tangent vector $\|\frac{d}{ds}{\gamma}(s)\|$ equals one. In the following, we denote the derivative of a regular curve $\gamma$ with respect to the arc length parameter $s$ by $\gamma'(s)$. 

The normal $N(s)$ is the normalized vector orthogonal to the tangent $\gamma'(s)$, such that the vector set $(\gamma'(s), N(s))$ has the same orientation as the canonical basis of $\R^{2}$. The curvature $\kappa(s)$ of the curve $\gamma$ is then defined by 
\begin{equation}\label{eq Def Kruemmung}
\gamma''(s) = \kappa(s) N(s).
\end{equation}
Therefore, the curvature takes on positive as well as negative values.  
As a consequence, $$\kappa (s) = \inn{N(s)}{\gamma''(s)} = \pm \|\gamma''(s)\|.$$ 
Using the fact that the normal vector $N$ is orthogonal to the tangent vector of $\gamma$ yields $$\kappa(s)=- \inn{N'(s)}{\gamma'(s)}.$$

For each planar Frenet curve with non-vanishing curvature $\kappa \neq 0$ the circle with center
$$
\gamma(s_{0}) + \frac{1}{\kappa(s_{0})} N(s_{0})
$$ 
and radius
$$
R(s_{0}) = \frac{1}{\kappa(s_{0})}
$$
is called osculating circle $K(s_{0})$ of $\gamma$ at $s_{0}$.
$K(s_{0})$ approximates $\gamma$ at the point $s_{0}$ of second order $o(s^{2})$ for $s\to s_{0}$ and therefore is unique due to the Frenet equalities, see e.g. \cite{docarmo}. 
As a consequence, the curvature of a regular curve at the point $\gamma(s_{0})$ depends only on the graph $\gr\gamma$ and not on the particular parametrization.

We consider such a  curve $\gamma: [s_{1},s_{2}]\to \R^{2}$ parametrized by arc length $s$, and transform it into local polar coordinates with respect to the angle $\phi: [s_{1},s_{2}] \to \R$, such that the tangent and the normal satisfy
\begin{eqnarray}
\gamma'(s) & = &(\sin \phi(s), -\cos\phi(s))^{T}, \nonumber \\  
N(s) &=& (\cos \phi(s), \sin \phi(s))^{T}, \nonumber
\end{eqnarray}
and such that the curvature $\kappa$ doesn't vanish.
In this setting, the following formula for the curvature holds:
\begin{eqnarray}
\kappa(s_{0})& =& - \skpl N'(s_{0}), \gamma'(s_{0}) \skpr 
\nonumber \\
& =& - \skpl (- \sin \phi(s_{0}), \cos \phi(s_{0}) )^{T} , \gamma'(s_{0}) \skpr\,  \frac{d\phi}{ds}(s_{0})
\nonumber\\
&= & \frac{d\phi}{ds}(s_{0})
\label{eq Kruemmung in s}
\end{eqnarray}
for every $s_{0}\in\,  ]s_{1},s_{2}[$. 
Denote $h: [\phi_{1},\phi_{2}]\to [s_{1},s_{2}]$ the diffeomorphism that transforms polar coordinates into arc length coordinates. Then by \eqref{eq Kruemmung in s} 
$$
\kappa(h(\phi)) =  (h^{-1})'(h(\phi)).
$$
Hence by the derivative of the inverse function
\begin{eqnarray}
\int_{s_{1}}^{s_{2}}ds 
 &= & \int_{\phi_{1}}^{\phi_{2}} |h'(\phi)| d\phi
 \nonumber\\
 & = & \int_{\phi_{1}}^{\phi_{2}} \frac{1}{|(h^{-1})'(h(\phi))|}\, d\phi 
 \nonumber\\
 &= & \int_{\phi_{1}}^{\phi_{2}}  \frac{1}{|\kappa(h(\phi))|}\, d\phi.
 \label{eq trmcurv}
\end{eqnarray}
As a special case, the entire length of the curve is given by
$$
L (\gamma) = \int_{0}^{L(\gamma)}ds = \int_{h^{-1}(0)}^{h^{-1}(L (\gamma))} \frac{1}{|\kappa(h(\phi))|}\, d\phi.
$$

\section{Image model: Segmented images and their boundaries}
\label{sec imaging}

We assume that we have an already segmented scene coming from some real-world image, i.e., we are considering continuously defined black and white images
\begin{align}
I: [\,0,1\,]^{2} \to \{0,1\}.
\label{eq Kontinuierlich definiertes Bild}
\end{align}
We consider discretizations into \emph{pixels} (or rather \emph{grid squares}) $\{\Omega_{h}(i,j)\}_{i,j}$
%\begin{align*}
%\Omega_n(i,j):=\left[\frac{i-1}{n},\frac i n\right]\times\left[\frac {j-1}{n},\frac j n\right],\quad 1\leq i,j\leq n,
%\end{align*}
%of size $\frac 1 n \times \frac 1 n$ for the \emph{grid resolution} $n\in\N$. 
of size $h \times h$, i.e., for the grid resolution $\frac 1 h>0$. In our case, the resulting discretized images are digital images, i.e., quantized. 
The center of a pixel is called  the \emph{grid point} and is denoted by $p_{h}(i,j)$.
\begin{align*}
N_4 &(p_h(i,j)) := \\
&\{p_h(i+1,j),p_h(i,j+1),p_h(i-1,j), p_h(i,j-1)\}
\end{align*}
is called 4-neighborhood of the grid point $p_h(i,j)$. We call two grid points $p$ and $q$ neighbors, if $p\in N_4(q)$, or symmetrically, $q \in N_{4}(p)$. In the same way, the 4-neighborhood of the corresponding pixels is defined \cite{klettebuch}.

A digital path is a sequence $p_{0}, \ldots p_{n}$ of pixels  (or a sequence of grid points), such that $p_{i}$ is  neighbor of $p_{i-1}$ for all $1 \leq i \leq n$.

A subset $P\subset\Omega_h$ of the digital image is called \emph{connected}, if every pair of pixels $p,q \in P$ can be joined by a digital path which lies completely in $P$. Maximal connected sets are called \emph{connected components}.

The sides of a pixel are called \emph{grid edges}, and the corners of each pixel are called \emph{grid vertices}.
The boundary of a connected component $P$ is the set of grid edges, which belong to a pixel in $P$, as well as to a pixel not in $P$.

To discretize the continuously defined image $I$ in (\ref{eq Kontinuierlich definiertes Bild}) we consider the \emph{Gau\ss\ discretization} $D_{h}(I)$, which is the union of the pixels in the descretization of resolution $\frac 1 h>0$ with center points in $I$ \cite{klettebuch}. In the same way, we say that a pixel $p$ belongs to a connected component $S\subset I$, if its center point lies in $S$ (see Fig. \ref{object}).

%---------------------------------------------------------------------------

\begin{figure}[h]
\begin{center}
\includegraphics[width=0.4\textwidth]{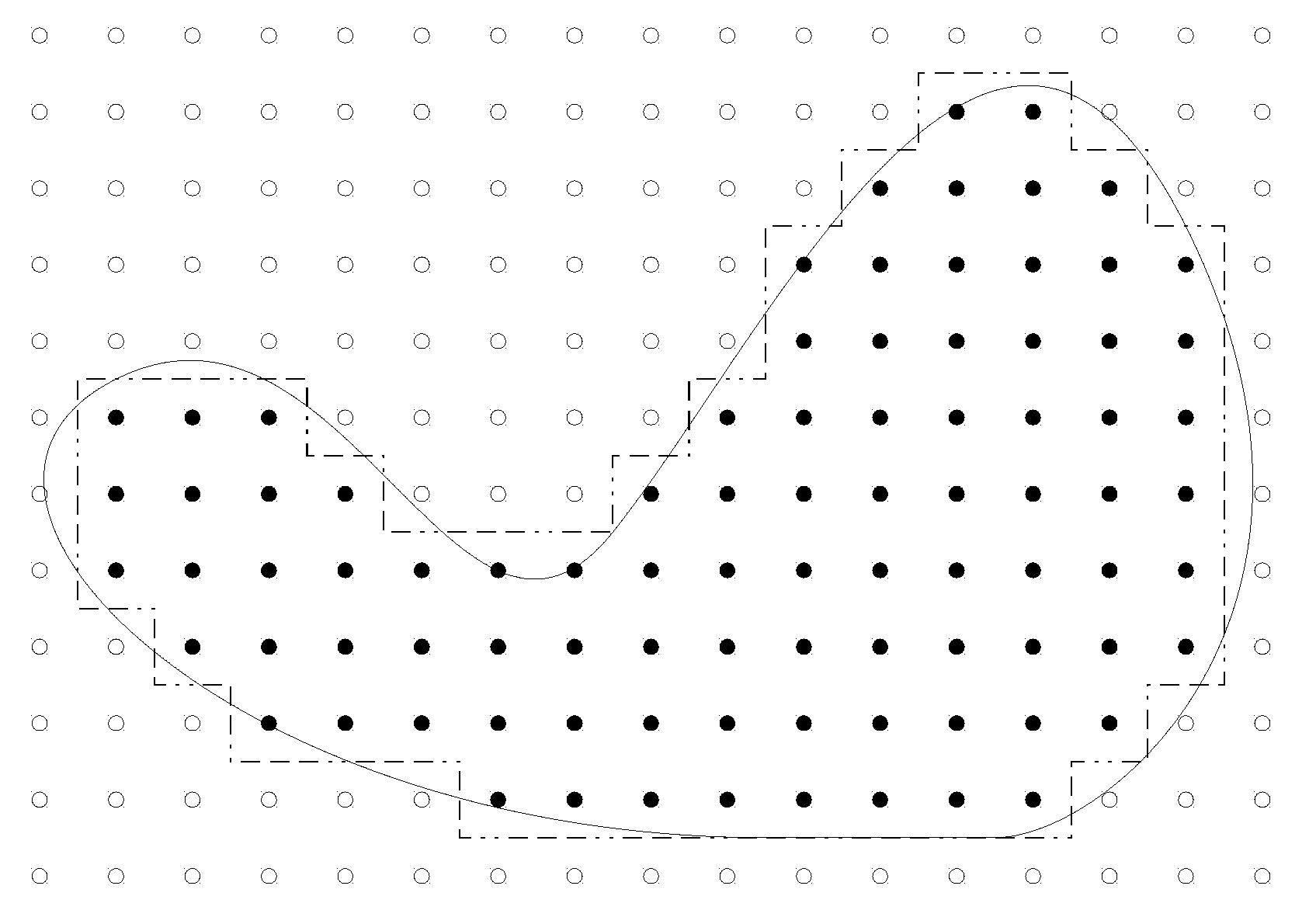}
\caption{An object bounded by the black line is discretized into pixels. The center points of the pixels are depicted by dots. The dots $\bullet$ are interior center points with respect to the object, the dots $\circ$ are the exterior center points. The boundary of the discretized object is the dashed line of grid edges.}
\label{object}
\end{center}
\end{figure}

\section{Formulation of the main task: Optimal estimation of the curvature profile}
\label{sec Main Task}

%---------------------------------------------------------------------------
We now consider a digital image, which is a discretized version of the continuous, already segmented image $I$ in \ref{eq Kontinuierlich definiertes Bild}. We assume that we are interested in a certain object in that segmented image.
%in which our object of interest is already segmented. 
The real, exact object is given by the closed set $X\subset[\,0,1\,]^2$. 
For example, $X$ can be a segmented area or a curve. Our objective is the estimation of the curvature profile along the boundary $\partial X$ of $X$, which is the curvature along the curve $\gamma:[a,b]\rightarrow[\,0,1\,]^2$ with $\gamma([\, a,b\,])=\partial X$. 
If $X$ is a nonempty, bounded, and convex subset of $\mathbb{R}^2$ with boundary $\partial X$, and $\gamma$ a curve with $\gr \gamma\subset \partial X$, then $\gamma$ is called convex.

In the following we define the curvature at the point $x=\gamma(t)$ with respect to the set $X$ by $\kappa(x):=\kappa(X,x):=\kappa(t)$.
% If the set $X$ is clear we only write $\kappa(x)$. 
Then we can pose the problem as follows:
%\begin{block}{Problem}
\begin{itemize}
\item \textbf{Assumptions:} We are given a nonempty, connected set of pixels $$D_h(X)\subset\Omega_h:=\{\Omega_h(i,j)\}_{i,j \subset K\times J \subset \Z^{2}},$$ which is achieved from the Gau\ss\ discretization of the object $X$ with resolution $1/h$.
\item\textbf{Objective:} Optimal estimation of the curvature profile $\kappa(X,x),\,x\in\partial X$, along the boundary $\partial X$ of the exact object $X$.
% under the
%\item \textbf{Assumption:} $P_n(X)$ is given by the Gauss discretization of the exact object $X$.
% i.e. $P_n(X)=D_n(X)$.
\end{itemize}

\section{The MDCA curvature estimator}
\label{MDCAchap}
%---------------------------------------------------------------------------
\subsection{Definition and Basic Properties of the Maximal Digital Circular Arc (MDCA)}
\label{MDCAchap2}

We consider the Minkowski or city-block metric $d_{1}$: For two points $x = (x_{1},x_{2})$ and $y = (y_{1},y_{2})$, the metric is defined as
$$
d_{1}(x,y) := |x_{1}-y_{1}| + |x_{2}-y_{2}| =: \| x-y\|_{1}.
$$
In fact, on $\R^2$ and its subsets, the metric is a norm.

Let $(v_k)_{k\in\{1,\dots,N+1\}}$ be a finite sequence of grid vertices with
$\|v_k-v_{k+1}\|_{1}= h$ and $v_{k+2}\neq v_k$. 
We call $C=(e_k)_{k\in \{1,\dots,N\}}$ with
$e_k=[v_k,v_{k+1}]$ a \emph{digital curve} (with parameter $k$) of length $N$. 
$e_{1}$ is called initial edge of the digital curve; $e_{N}$ is the terminal edge.

The \emph{distance} $d$ between two edges $e_i$ and $e_j$ of a digital curve $C=(e_k)_{k=1,\dots,N}$ is defined by
\[
d(e_i,e_j):=|j-i|
\]
If $C$ is simple (i.e., branching index at most 2 for each point) and closed, we define
\[
d(e_i,e_j):=\min(|j-i|,N-|j-i|)
\]
Then obviously $d(e_1,e_N)=1$.

We assume that the boundary $\partial D_h(X)$ of the Gau\ss\ discretized object $X$ in our image is a simple closed digital curve $\partial D_h(X)$. 
Each edge $e_k\in \partial D_h(X)$ is located between two pixels, an inner pixel $p\in D_h(X)$ and an outer pixel $q\notin D_h(X)$.
Note that the digital curve is needed to be closed only to have the orientation in inner and outer pixels. Since digital curves can be closed in adding appropriate edges, our considerations are not limited to the case of closed curves. 

The make use of the following notions of circular separability and digital circular arcs from \cite{rouss}: 

\begin{definition}(Circular separable)
\label{D circular separable}
Two sets of points $P=\{p_1,\dots p_N\}\subset\mathbb{R}^2$ and $Q=\{q_1,\dots,q_M\}\subset \mathbb{R}^2$ are called \emph{circular separable}, 
if there is a radius $R\geq 0$ and $m\in \mathbb{R}^2$, such that for all $p\in P$ and $q\in Q$ it holds that
$\|m-p\|\leq R$ and $\|m-q\|\geq R$ (or $\|m-p\|\geq R$ and $\|m-q\|\leq R$).
\end{definition}

\begin{definition}\label{sep} (Digital Circular Arc DCA)
Each subset $B=\{e_k\in \partial D_h(X)  \mid k=1,\dots, N\} \subset \partial D_h(X) $ whose outer pixels and inner pixels are circular separable, is called \emph{digital circular arc} (DCA). 
\end{definition}

For our purposes we refine the notion of the digital circular arc in the following way:

\begin{definition}(Maximal Digital Circular Arc MDCA)
A digital circular arc $A\subset \partial D_h(X) $ is called \emph{maximal digital circular arc} (MDCA), if every set $B \subset \partial D_h(X)$ that contains $A$ -- i.e. $A \subset B \subset \partial D_h(X)$ -- is not circular separable.
\end{definition}

\begin{figure}[h]
\begin{center}
\includegraphics[width=0.4\textwidth]{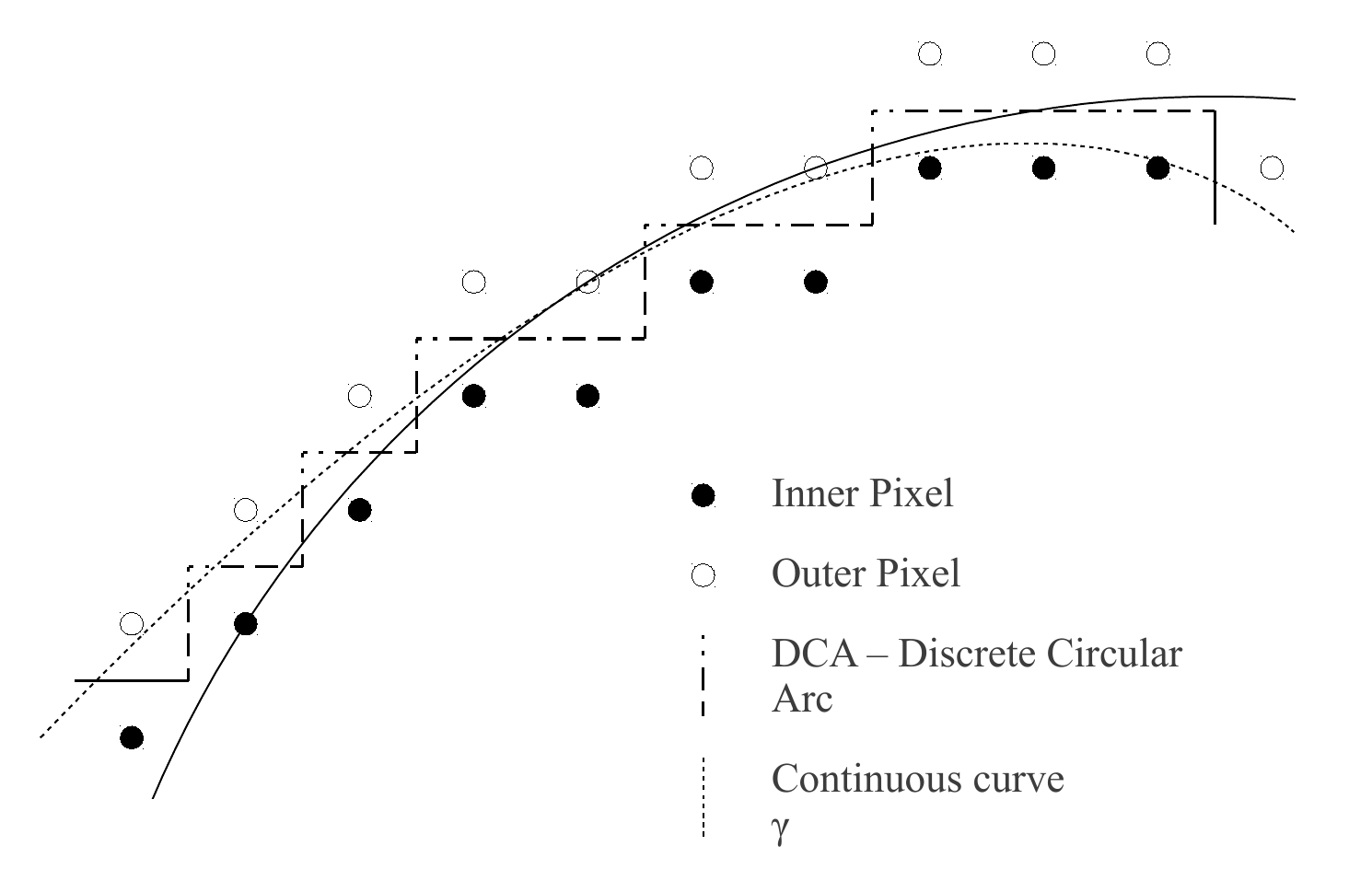}
\caption{
The digital circular arc (dashed) for the continuous curve $\gamma$ (dotted) separates into inner and outer pixels. For comparison, the separating circle is plotted as continuous curve.
}\label{dcagraph}
\end{center}
\end{figure}

Maximal digital circular arcs satisfy the following properties:

\begin{lemma}\label{MDC}\cite{rouss}
\begin{itemize}
\item[(i)] The set of MDCAs of a given digital curve $C$ is unique.
\item[(ii)] Different MDCAs have different sets consisting of the initial and terminal edges. 
\end{itemize}
\end{lemma}

%\begin{remark}
As a consequence, the MDCAs of a digital curve $C = (e_{k})_{k=1,\ldots,N}$ can be ordered with respect to their initial edges, where the parametrization $k$ of the digital curve gives the ordering. 
We denote the ordered sequence of the MDCAs by $(A_l)_{l=1,\dots M}$, $M \in\N$.
%\end{remark}

%---------------------------------------------------------------------------
\subsection{Computation of MDCAs and sets of MDCAs}
%\label{ssec Caluculation of the MDCAs}
\label{ssec Computation of MDCAs and sets of MDCAs}

To compute the set of MDCAs one first has to determine for a given digital curve $C$ the circular digital arcs. The basic idea of how to do this is as follows: Starting with an edge $e$ in $C$ one computes a DCA and extends it forward and backward along the edges of the curve $C$ until the DCA becomes maximal. The thusly obtained DCA is the first MDCA in the set of MDCAs of $C$. All further elements are computed by choosing the first edge in $C$ which is not contained in a previous MDCA and computing the maximal DCA from initial edge onwards.

There exist several optimized algorithms to obtain sets of MDCAs. For a detailed discussion we refer to \cite{rouss_three_versions} and the references given therein.

%---------------------------------------------------------------------------
\subsection{MDCA curvature estimator}

Let be given a DCA $A=(e_k)_{k=1,\dots,n}$ in $\partial D_{h}(X)$. We denote by $R^h(A)$ 
the radius of the smallest separating circle, for which the conditions of Definition \ref{D circular separable} holds.
We denote by 
\[
k^h(A):=1/R^h(A),
\] 
the {\em discrete curvature},  and by 
\[
e(A):=e_{\lceil n/2 \rceil}
\]
the {\em central edge} of $A$. Here $\lceil \cdot \rceil$ denotes the ceiling function.

\begin{definition}\label{mdcaest}\cite{rouss} (MDCA curvature estimator)
Let $(A_l)_{l=1,\dots,M}$ be a sequence of MDCAs of a digital curve $C$.
Suppose $E_l$ is a partition of $C$ where $E_{l}$ is the set of edges $e\in C$ such that
$$
d(e,e(A_l))\leq d(e,e(A_k)) \quad \forall k\in\{1,\dots,N\}.
$$

The MDCA curvature estimator $\widehat{\kappa}^h_{\text{MDCA}}(C,e)$ for the digital curve $C$ and an edge $e\in E_l$ in $C$ is defined as the piecewise constant function 
\[
\widehat{\kappa}^h_{\text{MDCA}}(C,e):=k^h(A_l).
\]
\end{definition}

For the algorithm, for a given simple closed digital curve $C=(e_k)_{k=1,\dots N}$ and grid size $h>0$ we first compute the set of MDCAs $(A_l)_{l=1,\dots,M}$ as described in \cite{rouss_three_versions}.
Denote $R_l^{h}$ the radius of the smallest separating circle of $A_l$ expressed in terms of the unit $h$, and let $$M_l:=e_{\lceil (j_l-i_l)/2 \rceil \!\!\!\!\mod N}$$
be the set of middle grid edges.
For all edges $e_{k}$ in $C$, $k=1,\dots,N$, we determine the middle grid edge $e_{i}\in \{ M_{l} \mid l=1, \ldots M\}$ which is closest to $e_{k}$.  
Then we set $$\widehat{\kappa}^h_{MDCA}(C,e_k)=1/R_i^{h}.$$

\subsection{$\lambda$-MDCA curvature estimator}

We shortly would like to remark on another possible curvature estimator, which we will call $\lambda$-MDCA curvature estimator. The idea is to consider not only the closest MDCA for an edge $e_{k}$ of a given curve $C$, but rather all MDCAs in which $e_{k}$ is contained. The weighting of the contributing curvature estimates is chosen as in \cite{lac}, where a similar construction was considered for the estimation of tangents. 

\begin{definition}
Let $C=(e_k)_{k=\{1,\dots,N\}}$ be a digital arc and  let $(A_l)_{l\in\{1,\dots,M\}}$ be the set of MDCAs of $C$. Denote by $e_{i_l}$ and $e_{j_l}$ the start and terminal edges of the MDCA $A_l$. 
For an edge $e_{k}$ lying on $A_{l}$ we define 
$$E_l(k):=\left|\frac {k-i_l}{j_l-i_l}\right|$$
as the eccentricity of $e_{k}$ with respect to $A_{l}$.
\end{definition}

Obviously, $E_l(k)\in[0,1]$. The closer $E_l(k)$ approaches $1/2$, the closer the edge $e_k$ is lying to the middle edge of the MDCA $A_l$. 

The idea is now to weight the contribution of the curvatures of the  various MDCAs with respect to the eccentricity of the considered edge $e_{k}$ of $C$. To this end, we consider a continuous concave function $\lambda:[0,1]\rightarrow\mathbb{R}^{+}_{0}$ which takes its maximum at the point $1/2$ and vanishes at the endpoints of the interval. For example, the functions $\lambda(t):=4t(1-t)$ or $\lambda(t):=-t\log(t)-(1-t)\log(1-t)$ fulfill these requirements.

\begin{definition}\label{lam}($\lambda$-MDCA curvature estimator) 

Let $C=(e_k)_{k=\{1,\dots,N\}}$ be a digital curve, $(A_l)_{l\in\{1,\dots,n\}}$ be the set of MDCAs, and  $\mathcal{P}(e_k):=\{l:e_k\in A_l\}$. Let $k^h(A)$ be defined as in Definition \ref{mdcaest}.

Then
\[
\widehat{\kappa}^h_{\lambda\text{-MDCA}}(C,e_k):=\frac{\sum_{l\in\mathcal{P}(e_k)}\lambda(E_l(k))k^h(A_l)}{\sum_{l\in\mathcal{P}(e_k)}\lambda(E_l(k))}
\]
is the  $\lambda$-MDCA curvature estimator of the digital curve $C$. 
\end{definition} 

In the following section, we prove the convergence of the MDCA curvature estimator, but leave the convergence of the $\lambda$-MDCA curvature estimator as an open problem. 
In Section \ref{ergeb}, where we discuss experimental results, we will use the $\lambda$-MDCA curvature estimator for comparison purposes. According these experimental results we conjecture convergence of this weighted estimator with rate comparable to the classical MDCA curvature estimator.

\section{Multigrid convergence of the MDCA curvature estimator}\label{errest}

In order to obtain error estimates for the MDCA curvature estimator, we first require the concept of multigrid convergence. Multigrid convergence allows the quantification of the error of a discrete geometric estimator when the pixel width gets arbitrarily small, i.e.,  $h\to 0$.

Discrete geometric estimators are used both to determine global quantities, such as the length of a curve or the area of a surface, as well as local quantities, such as the tangent direction or the curvature of a curve. To the best of our knowledge, up to now, the multigrid convergence of a curvature estimator \cite{multi} is an open problem. First estimates are considered in \cite{rouss} where additional assumptions such as restrictions on the growth of the discrete length of the MDCAs involved. Here, we provide a proof for the uniform multigrid convergence of the MDCA curvature estimator without these assumptions.

%---------------------------------------------------------------------------
\subsection{Multigrid convergence}
We define the multigrid convergence for local, discrete, geometric estimators.

Let $\mathbb{X}$ be a family of nonempty, compact, simply connected subsets of $\mathbb{R}^2$ and let $\partial D_h(X)$ be a digital curve. A {\em local, discrete, geometric estimator} $\widehat{E}=(\widehat{E}^h)_{h>0}$ of a local, geometric feature $E(X,x)$ of a set $X$ at the point $x\in\partial X$ is a family of mappings, which assign a real number to a digital curve and an edge on this curve.
\begin{definition}
A local, discrete, geometric estimator $\widehat{E}$ is called {\em multigrid convergent} for the family $\mathbb{X}$, if for all $X\in\mathbb{X}$, $h>0$ and $x\in\partial X$ the following condition holds:
\begin{gather*}
\forall e\in\partial D_h(X),\;\forall y\in e \text{ with }\|x-y\|_1\leq h\\
\exists\, \tau_x(h)\text{ with }\lim_{h\rightarrow 0}\tau_x(h)=0\text{ such that}\\
|\widehat{E}^h(\partial D_h(X),e)-E(X,x)|\leq \tau_x(h).
\end{gather*}
$\widehat{E}$ is called {\em uniform multigrid convergent} if there exists a mapping $\tau(h)$, independent of $x$, with $\tau(h)\geq\tau_x(h)$. $\tau(h)$ is then called the {\em convergence speed}.
\end{definition}
In the following, we take $\widehat{E}$ to be the curvature estimator $\widehat{\kappa}_{\text{MDCA}}$, and in addition require that for any $X\in\mathbb{X}$ the mapping $\partial X\rightarrow \mathbb{R}$, $x\mapsto \kappa(X,x)$, is continuous. 

%---------------------------------------------------------------------------
\subsection{Proof of multigrid convergence of the MDCA curvature estimator}

Firstly, we define the support function of a convex curve which provide a unique characterization of the curve.
\begin{definition}(Support function)\\
\label{D support function}
Let $\,\gamma:[a,b]\rightarrow \mathbb{R}^2$ be a convex Frenet curve with respect to polar coordinates as in Section \ref{diffgeo}, and let
\[
N(\phi):=\left(\begin{array}{c} \cos \phi\\ \sin \phi\end{array}\right)
%\quad\text{and}\quad
%t(\phi):=\left(\begin{array}{c} -\sin(\phi)\\ \cos(\phi)\end{array}\right).
\]
be its normal.
The {\em support function $f:=f_\gamma^M$ of $\gamma$ with center $M\in \R^2$} is defined by
\begin{gather*}
f_{\gamma}^M(\phi):\;[\phi_1,\phi_2]\rightarrow \mathbb{R},\\
f_{\gamma}^M(\phi):=
\max\left\{\left<\gamma(t)-M,N(\phi)\right>:t\in [a,b]\right\}
\end{gather*}
Here $\phi_1$ and $\phi_2$ is the angle between $\gamma(a)-M$, resp., $\gamma(b)-M$ and the $x$-axis. 
\end{definition}

Note that for predefined $M$ the support function $f_{\gamma}^M(\phi)$  determines the curve $\gamma$ completely \cite{suppfun}.

In the following, we write $f'(\phi) : = \left(f_{\gamma}^M\right)'(\phi):=\frac d {d\phi}f_{\gamma}^M(\phi)$ and 
$f''(\phi) :=\left(f_{\gamma}^M\right)''(\phi):=\frac {d^{2}} {d\phi^{2}}f_{\gamma}^M(\phi)$ for the derivatives of the support function. 

\begin{lemma}\label{traeger}
Let $\gamma:[a,b]\rightarrow \mathbb{R}^2$ be a convex curve as in Definition \ref{D support function}. 
%with normal $N(\phi)$. 
Let $L = (\theta_{2}-\theta_{1})R$ be the central length of the ring segment.
Suppose that $\bar{x} = \gamma(\phi)\in \gr\gamma$. Then,
\[
\frac 1 {\kappa(\bar{x})}=f_{\gamma}^M(\phi)+\left(f_{\gamma}^M\right)^{\prime\prime}(\phi).
\]
\end{lemma}
For the proof, we refer to \cite[p. 583]{suppfun}.

\begin{definition}
For $M\in\mathbb{R}^2$, $R, d\in\mathbb{R}^{+}$, and $\theta_1$,$\theta_2\in[0,2\pi]$, we define a
{\em $(M, R, d, \theta_1, \theta_2)$-ring segment} by 
\begin{align*}
\mathcal{R} &:= \mathcal{R}(M, R, d, \theta_1, \theta_2)\\
&:= \left\{M+(r\cos\theta,r\sin\theta)^{T}|R-d\leq r\leq R+d,\theta\in [\theta_1,\theta_2]\right\}.
\end{align*}
If for all $\theta\in[\theta_1,\theta_2] $, the intersection
$$\{M+r(\cos\theta,\sin\theta)^{T}|r>0\}\cap \gr \gamma $$ 
with a convex curve $\gamma$ consists exactly of one point $x(\theta)\in\mathcal{R}$, then one says that the ring segment \emph{simply covers} $\gamma$.
\end{definition}

Next, we extend Lemma 1 in \cite{rouss} in the sense that we explicitly give the order of approximation and provide easily verifiable conditions on the convex curve $\gamma$.
To this end, let $\mathcal{R} = \mathcal{R}(M, R, d, \theta_1, \theta_2)$ be a ring segment simply covering a convex curve $\gamma$. We consider a new parameterization of $\gamma$ in terms of $\theta\in [\theta_1,\theta_2]$:
$$
x(\theta):=\{M+r(\cos\theta,\sin\theta)^{T}|r>0\}\cap \gr \gamma.
$$ 
For $i=1,2$ the point $x(\theta_i) = \gamma(\phi_{i})$ has normal vector $N(\phi_i)$ and $\gr \gamma\cap\mathcal{R}$ is uniquely characterized by the support function $f=f_{\gamma}^M: [\phi_1,\phi_2]\rightarrow \mathbb{R}$. (See Figure \ref{ringgraph}.)

\begin{figure}[h]
\begin{center}
\includegraphics[width=0.49\textwidth]{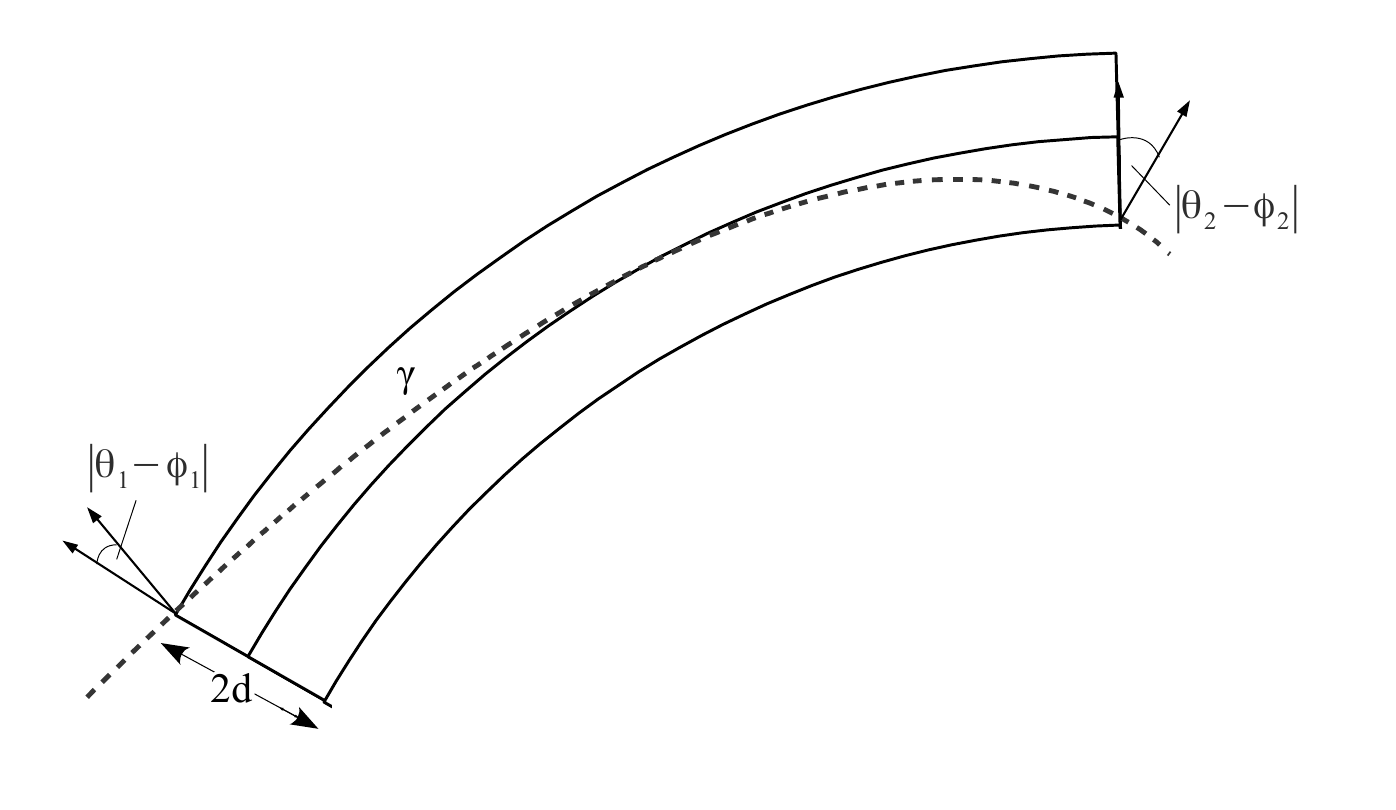}
\caption{Simply covering ring segment for the curve $\gamma$. Note the difference of the angles $|\theta_{1}-\phi_{1}|$ and $|\theta_{2}-\phi_{2}|$ of the segment and the normal to the curve $\gamma$.}\label{ringgraph}
\end{center}
\end{figure}

\noindent
In order to state and prove the main theorem, we require some lemmas.

\begin{lemma}\label{lem12}
Let $\gamma$ be a convex {Frenet}  curve and let $\cR$ be a ring segment simply covering $\gamma$. Suppose that the curvature $\kappa$ of $\gr\gamma\cap\cR$ is bounded  below by $\kappa_{\min} > 0$ and above by $\kappa_{\max}$. Then the angular difference $\Delta\phi := \phi_2-\phi_1$ is bounded as follows:
\be\label{phi}
\kappa_{\min}L(R-d)/R\leq \Delta\phi\leq \kappa_{\max}(L(R+d)/R+4d).
\ee
\end{lemma}
\begin{proof}
Let $L_{\gamma} = L_{\gamma}(x(\theta_{1}),x(\theta_{2}))$ be the length of $\gamma$ from $x(\theta_1)$ to $x(\theta_2)$. The convexity of $\gamma$ implies that
$$
(R-d)(\theta_2-\theta_1) \leq L_\gamma\leq (R+d)(\theta_2-\theta_1)+4d,
$$
which is equivalent to
$$
L(R-d)/R \leq L_\gamma\leq L(R+d)/R+4d.
$$
By Eqn. \eqref{eq trmcurv} and the convexity of $\gamma$ (i.e. $\kappa > 0$)
we have that 
$$
L_\gamma=\int_{\phi_1}^{\phi_2}\frac {1}{ \kappa(\gamma(\phi))} d\phi.
$$
As the curvature $\kappa$ is bounded above and below we obtain
$$
\kappa_{\min}L(R-d)/R\leq \Delta\phi\leq \kappa_{\max}(L(R+d)/R+4d). \qedhere
$$
\end{proof}
The next lemma gives bound on the support function $f(\phi)$.

\begin{lemma}\label{lem13}
Let $\gamma$ be a convex {Frenet} curve and let $\cR$ be a ring segment simply covering $\gamma$. Suppose that the curvature $\kappa$ is bounded above by $\kappa_{\max} > 0$. Then the support function $f(\phi)$ of $\gamma$ satisfies the estimate
\begin{align}\label{cos2}
(R-d)\big(1-2d\kappa_{\max}\frac{R}{R+d}\big)\leq f(\phi)\leq R+d.
\end{align}
\end{lemma}
\begin{proof}
For each $x(\theta)\in \gr\gamma\cap\mathcal{R}$ with normal $N(\phi)$, we have that $f(\phi) = \|x-M\|\cos |\phi-\theta| $. Consequently,
\begin{align}\label{cos1}
(R-d)\cos(|\phi-\theta|)\leq f(\phi)\leq (R+d).
\end{align}
The maximum angle $|\phi-\theta|_{\max}$is obtained if the curve $\gamma$ maintains constant maximum curvature $\kappa_{\max}$ and remains in $\mathcal{R}$; see Figure \ref{circle}.

\begin{figure}[h]
\begin{center}
\includegraphics[width=0.45\textwidth]{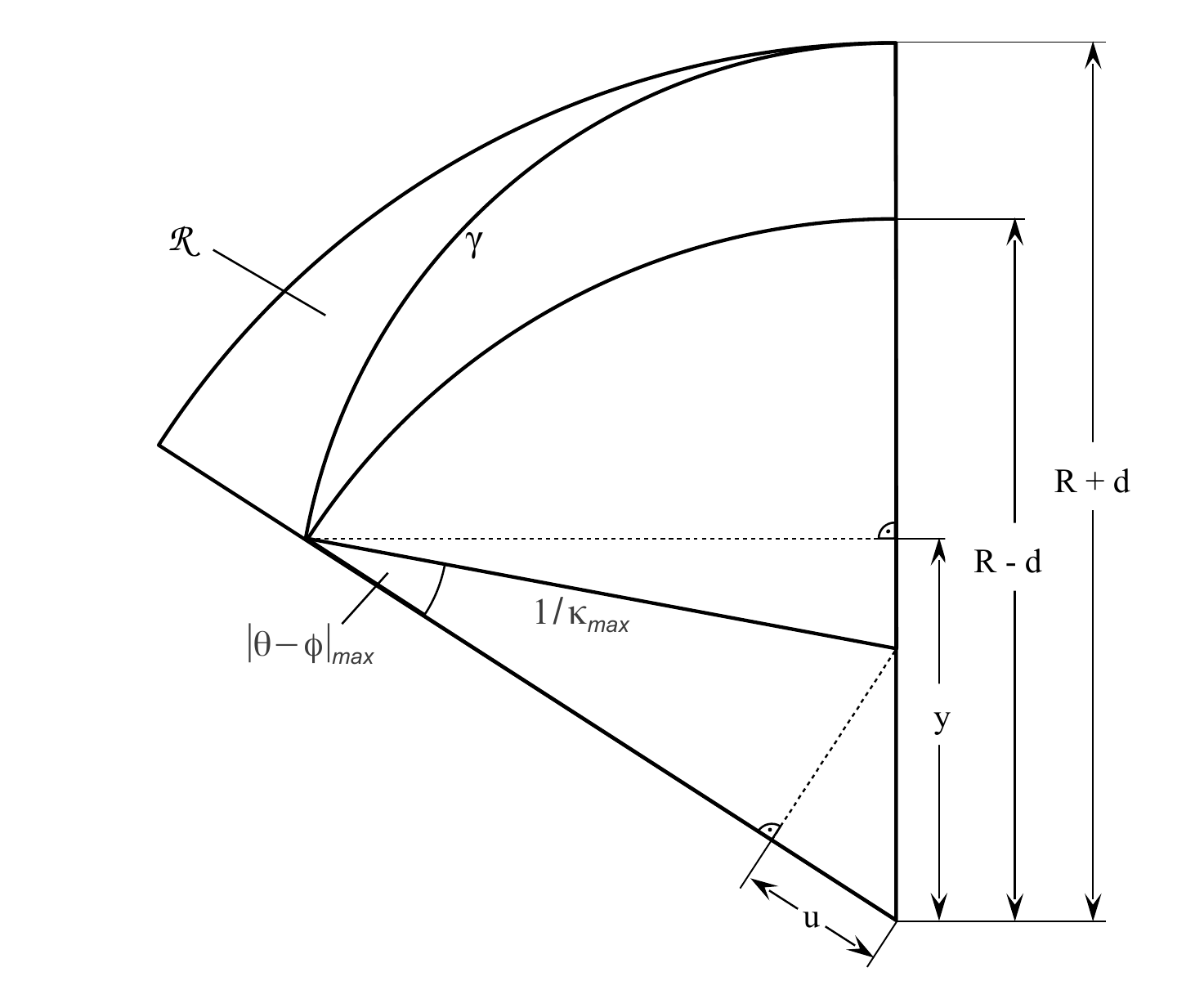}
\caption{The curve $\gamma$ in the ring segment $\mathcal{R}$ and the properties of inner and outer radii.}\label{circle}
\end{center}
\end{figure}

Thus,
\[
\cos(|\phi-\theta|_{\max})=(R-d-u)\kappa_{\max}
\]
and 
\[
\frac u {R+d-1/\kappa_{\max}}=\frac y {R-d}.
\]
Together with 
\[
(R-d)^2-y^2=1/\kappa_{\max}^2-(y-(R+d-1/\kappa_{\max}))^2,
\]
this yields
\begin{align*}
\cos(|\phi-\theta|_{\max})=1-2Rd\kappa_{\max}/(R+d).
\end{align*}
Substitution into (\ref{cos1}) produces
$$
(R-d)(1-2Rd\kappa_{\max}/(R+d))\leq f(\phi)\leq R+d. \qedhere
$$
\end{proof}

We also need an estimate about the asymptotic behavior with respect to the  length $L_\gamma$ and the thickness $d$. This estimate is provided by

\begin{lemma}\label{lem14}
Assume that $\gamma$ is an at least four times differentiable convex curve with ring segment $\cR$ which simply covers $\gamma$. {Let $ \gamma$ be parametrized by arc length $s$ in polar coordinates:
$$\gamma(s)=(r(s)\cos \theta(s) ,r(s) \sin \theta(s) )^{T},$$
and let $\gamma(0)\in\mathcal{R}$, for all $d>0$.} Suppose that the circle $\mathcal{K}(M,R)$ osculates the curve $\gamma$ at a point $x=\gamma(0)$, i.e., $r(0)=R$, $r'(0)=0$, and $\kappa(0)=1/R$. Further suppose that there exists an $\ell\in \N$ such that
\begin{align}
\kappa^{(k)}(0) &= 0, \quad\text{for $k = 1, \ldots, \ell - 3$,}\label{kabsch}\\
\kappa^{(\ell-2)}(0) &\neq 0, \quad\text{for $\ell \geq 3$.}
\end{align}
Then
\be\label{dL}
d \in \cO(L_\gamma^\ell), \quad\text{for $L_\gamma\to 0+$}.
\ee
\end{lemma}
\begin{proof}
{Let $\gamma$ be expressed in polar coordinates with arc length $s\in [-L_{1},L_{2}]$, where $L_{\gamma} = L_{1}+L_{2}$, compare with Figure \ref{l1l2graph}:}
\begin{align*}
\gamma(s)&=(r\cos \theta ,r \sin \theta )^{T},\\
\gamma'(s)&=(r'\cos \theta-r\theta'\sin\theta,r'\sin \theta+r\theta'\cos\theta)^{T},\\
\gamma''(s)&=\left(\begin{array}{c}r''-r\theta'^2\\2r'\theta'+r\theta''\end{array}\right)\cos\theta+
\left(\begin{array}{c}-2r'\theta'-r\theta''\\r''-r\theta'^2\end{array}\right)\sin\theta,
\end{align*}
which implies
\begin{align}\label{kap}
\kappa &=\inn{\gamma'' }{N(\theta)}=r \theta'^2-r''.
\end{align}
{Here, $\gamma, r, \theta$ and all their derivatives are functions in $s$.}

%Note that the following relation holds between $d$ and $L$:
\begin{figure}[h]
\begin{center}
\includegraphics[width=0.45\textwidth]{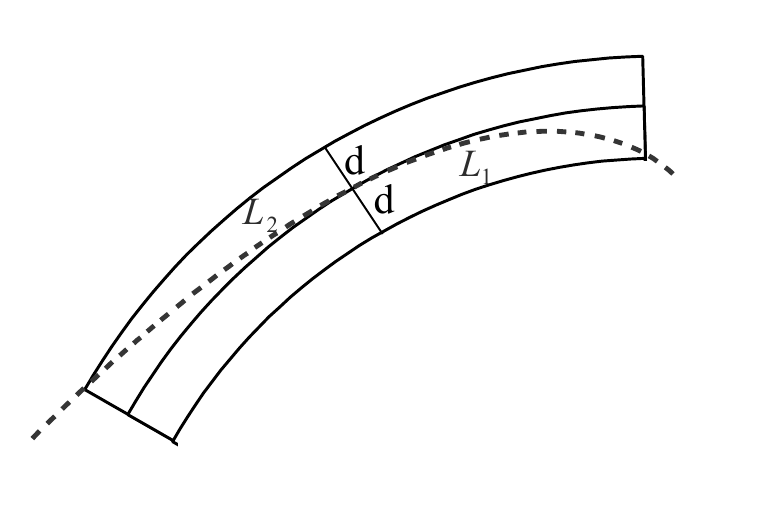}
\caption{{The distance $d$, and the segments $L_{1}$ and $L_{2}$ of the curve $\gamma$ (dashed) in the ring segment $\mathcal{R}$.}\label{l1l2graph}}
\end{center}
\end{figure}

Let $d(s) := R - r(s)$, {$-L_{1}\leq s \leq L_2$}. In order to use (\ref{kap}), we first need the derivatives of $\theta(s)$ at $s=0$. As $r'(0)=0$, {we have} $\theta(0)=\phi(0)$. Now,
\begin{align} 
\theta'(s)&=\frac{\text{d}\theta}{\text{d}s}(s)=\left(\frac{\text{d}\theta}{\text{d}s_\theta}\frac{\text{d}s_\theta}{\text{d}s}\right)(s)\nonumber\\
& =\frac1R\cdot\cos(\phi(s)-\theta(s))\quad \Rightarrow \quad \theta'(0)=\frac 1 R, \label{thet1}\\
\label{thet2}
\theta''(s)&=-\frac1R\cdot\sin(\phi(s)-\theta(s))(\phi'(s)-\theta'(s)),
\end{align}
implying $\theta''(0)=0$.

By induction on $n$ for $n\geq 2$ with initial induction step (\ref{thet2}), one can show that
\begin{align*}
R\cdot\theta^{(n)}(s)&=\cos(\phi(s)-\theta(s))\;\times\\
& \sum_{2\leq m\leq n-1\atop m \text{ even}}\;\prod_{1\leq i_k\leq n-2\atop\sum_{k=1}^m i_k=n-1}c_{i_k}(\phi^{(i_k)}(s)-\theta^{(i_k)}(s))\\
&+\sin(\phi(s)-\theta(s))\;\times\\
& \sum_{1\leq m\leq n-1\atop m \text{ odd }}\;\prod_{1\leq i_k\leq n-1\atop\sum_{k=1}^m i_k=n-1}d_{i_k}(\phi^{(i_k)}(s)-\theta^{(i_k)}(s)).
\end{align*}
{Condition (\ref{eq Kruemmung in s}) gives $\phi^{(\ell)}(s)=\kappa^{(\ell-1)}(s)$, and 
together with (\ref{kabsch}) we get}
\begin{align}\label{thet3}
&\theta^{(n)}(0)=0, \quad\text{ for }2\leq n\leq \ell-1.
\end{align}
Now we can calculate the derivatives of $d(s)$ at $s=0$:
\begin{align*}
d(0)&=R-R=0,\\
d'(0)&=-r'(0)=0,\\
d''(0)&=-r''(0)=\kappa(0)-\theta'^2(0)r(0)=\frac 1R-\frac{1}{R^2}R=0,
\end{align*} 
where the second-to-last equality holds because of (\ref{thet1}). The higher-order derivatives of $d$ are obtained again by induction on $n$ for $n\geq 2$ but with initial induction step $d''(s)=\kappa(s)-\theta'^2(s)r(s)$:
{
\begin{align*}
d^{(n)}(s) &=\kappa^{(n-2)}(s)\\
& -\sum_{\substack{1\leq i,j\leq n-1\\1\leq k\leq n-2\\i+j+k=n}}c_{ijk}\theta^{(i)}(s)\theta^{(j)}(s)r^{(k)}(s), \text{ for } c_{ijk}\in\mathbb{R}.
\end{align*}
}

Using \eqref{kabsch} and \eqref{thet3}, we thus obtain
\begin{align}\label{abd1}
&d^{(n)}(0)=0,\quad\text{ for }0\leq n\leq \ell-1,\\\label{abd2}
&d^{(\ell)}(0)=\kappa^{(\ell-2)}(0)>0.
\end{align}
Consequently, the Taylor expansion of $d(s)$ at $s=0$ is given by
{
\be\label{taylord}
d(s) = \frac{\kappa^{(\ell-2)}(0)}{\ell!}\, s^\ell+{o}(s^{\ell}),\mbox{ for } s \to 0.
\ee
}
In particular, since $|s| \leq \max\{L_{1},L_{2}\}\leq L_{\gamma}$, we have for $L_{\gamma}$ small enough
\[
d \in \cO (L_\gamma^\ell), \text{ for $\ell\geq 3$, as $L_\gamma \to 0+$}.
\]
\end{proof}

\begin{lemma}
Suppose the hypotheses of Lemmas \ref{lem12},\ref{lem13}, and \ref{lem14} are satisfied. Let $\bar{\phi} : = (\phi_{1}+\phi_{2})/2$. 

Then the second derivative of the support function $f$ obeys the following asymptotic estimate
\be\label{abl}
|f''(\bar{\phi})| \in\mathcal{O}\left(\frac d {L_\gamma^2}\right), \quad\text{as $L_\gamma$, $d\to 0+$}.
\ee
\end{lemma}
\begin{proof}
By induction on the order of derivative of $f$ with initial induction step $f^{(3)}=(1/\kappa)'-f'=\kappa'/\kappa^2-f'$ (cf. Lemma \ref{traeger}), one can show that for $n\geq 3$:
\begin{align*}
f^{(n)}=\left(\sum_{1\leq m\leq n-2}\kappa^{-(m+1)}\prod_{1\leq i_k\leq n-2,\atop 
 i_1 + \ldots + i _{m}=n-2}c_{i_km}\kappa^{(i_k)}\right)\\ \\
-f^{(n-2)},\qquad c_{i_km}\in\mathbb{R}.
\end{align*}
Using condition (\ref{kabsch}), one obtains
\be
f^{(k)} = -f^{(k-2)},\quad \text{ for all } 3\leq k\leq \ell-1,\\\label{zweite}
\ee
implying that
\[
f^{(k)} = \pm f'', \quad \text{ for all even } k\leq \ell-1.
\]
The Taylor expansion of $f$ at $\bar{\phi}=(\phi_1+\phi_2)/2$ gives
\begin{align*}
f(\phi_1)+f(\phi_2)=2f(\bar{\phi})+\frac{(\Delta\phi)^2} 4f''(\bar{\phi})\\
+\sum_{n=4\atop n \text{ even}}^{\infty}\frac{(\Delta\phi)^n}{2^{n-1}n!}f^{(n)}(\bar{\phi}),
\end{align*}
which together with (\ref{zweite}) implies
\begin{align*}
f(\phi_1)+ &f(\phi_2) = 2f(\bar{\phi})+\frac{(\Delta\phi)^2} 4f''(\bar{\phi})\\
&\pm\sum_{n=4\atop n \text{ even}}^{l-1}\frac{(\Delta\phi)^n}{2^{n-1}n!}f''(\bar{\phi})+\sum_{n=l\atop n \text{ even}}^{\infty}\frac{(\Delta\phi)^n}{2^{n-1}n!}f^{(n)}(\bar{\phi}).
\end{align*}

Consequently,
\begin{align}\label{expansion}
f''(\bar{\phi}) &= \frac{4(f(\phi_1)-2f(\bar{\phi})+f(\phi_2))}{(\Delta\phi)^2\pm\frac{1}{48}(\Delta\phi)^4\pm\dots} +
\mathcal{O}((\Delta\phi)^{\ell-2})\nonumber\\
&= \frac{4(f(\phi_1)-2f(\bar{\phi})+f(\phi_2))}{(\Delta\phi)^2 (1\pm o((\Delta\phi)^2)} +
\mathcal{O}((\Delta\phi)^{\ell-2}))\nonumber\\
&= \frac{4(f(\phi_1)-2f(\bar{\phi})+f(\phi_2))}{(\Delta\phi)^2}\, (1\mp o((\Delta\phi)^2)\nonumber\\
&\quad + \mathcal{O}((\Delta\phi)^{\ell-2})),\quad\text{for $\Delta\phi\to 0+$.}
\end{align}
Note that the lower bound in \eqref{phi} shows that $\Delta\phi\to 0+$ implies $L_\gamma\to 0+$, which by \eqref{dL} implies $d\to 0+$. Therefore, substituting (\ref{phi}) and (\ref{cos2}) into (\ref{expansion}), we obtain for $L_\gamma\to 0 +$ and $d \to 0 +$:
\begin{align*}
|f''(\bar{\phi})| &\leq 16\,R^2 \left(\frac{R+d+(R-d)\,R\kappa_{\max}}{\kappa_{\min}^2(R-d)^2(R+d)}\right)\;\times\\
&\quad\left(\frac{d}{L_\gamma^2}\right)\,\left(1 \mp o(L_\gamma^2)\right) + \mathcal{O}(L_\gamma^{\ell-2}).
%\subset \mathcal{O}\left(\frac d {L^2}\right)+\mathcal{O}(L^{\ell-2}).
\end{align*}
With $d\in\mathcal{O}(L_\gamma^{\ell})$ from (\ref{dL}), this gives
\begin{align*}
|f''(\bar{\phi})|&\in \mathcal{O}\left(\frac d {L_\gamma^2}\right)+\mathcal{O}(L_\gamma^{(\ell-2)})\nonumber\\
&\in\mathcal{O}\left(\frac d {L_\gamma^2}\right)\left(1+\mathcal{O}\left(\frac{L_\gamma^\ell}{d}\right)\right)\nonumber\\
&\in\mathcal{O}\left(\frac d {L_\gamma^2}\right)\left(1+\mathcal{O}(1)\right)
\in\mathcal{O}\left(\frac d {L_\gamma^2}\right).\qedhere
\end{align*}
\end{proof}

The next lemma is an extension of Lemma 1 in \cite{rouss}. Here, we explicitly give an order of approximation.
\begin{lemma}\label{lem16}
Assume that the hypotheses of Lemmas \ref{lem12}, \ref{lem13}, and \ref{lem14} are satisfied. Then the following estimate holds for all $x\in \mathcal{R} \cap \gr \gamma$:
\begin{align*}
|R-1/\kappa(x)|\in\mathcal{O}\left(\frac d {L_\gamma^2}\right), \text{ as } L_\gamma, d\rightarrow 0+.
\end{align*}
\end{lemma}
\begin{proof}
We use Lemma \ref{traeger}. Let $\bar{x}\in \gr\gamma\cap\mathcal{R}$ have normal $N(\bar{\phi})$. Then,
\begin{align*}
&1/\kappa(\bar{x})=f(\bar{\phi})+f''(\bar{\phi})\leq R+d+\mathcal{O}\left(\frac d{L_\gamma^2}\right)\\
\Rightarrow\quad&|1/\kappa(\bar{x})-R|\in\mathcal{O}\left(\frac d{L_\gamma^2}\right).
\end{align*}
By \eqref{kabsch}, we have that $L_\gamma\rightarrow 0$, and for all $x\in \gr\gamma\cap\mathcal{R}$ it follows that $x\rightarrow \bar{x}$, which completes the proof.
\end{proof}
%\end{description}

Employing the results from the above lemmas, we show that the length $L_\gamma$ of the curve $\gamma$ in the ring segment $\mathcal{R}$ does not converge too quickly to $0$ when the thickness $d\rightarrow 0+$. This then implies the convergence of our procedure.

\begin{lemma}\label{ring2}
Suppose the hypotheses of Lemmas \ref{lem13} and \ref{lem14} are satisfied. Then,
\[
d\in\mathcal{O}(L_\gamma^3), \quad\text{ for $L_\gamma\to 0+.$}
\]
\end{lemma}

\begin{proof}
First, we show that for all $d>0$, there exists an $x(\theta_0) = \gamma(\phi_{0})\in\mathcal{R}\cap \gr\gamma$ with normal $N(\phi_0)$, so that
$\sin(\theta_0-\phi_0)\in\mathcal{O}(d/L_\gamma)$ for $d \to 0 +$.

From Figure \ref{diff} one deduces that
\[
\sin |\theta-\phi| =\frac{dr}{ds_{\theta}},
\]
from which it follows that
\begin{align*}
\frac{L_\gamma(R-d)}R\,\sin(|\phi-\theta|_{\min})\leq L_\gamma\,\sin(|\phi-\theta|_{\min})\\
\leq\int_{\phi_1}^{\phi_2}\sin(|\phi-\theta|)\,ds_\theta=\int_{r_1}^{r_2}\,dr\leq 2d.
\end{align*}
Thus, for $d\to 0+$,
\begin{align}\label{sin}
\sin|\phi-\theta|_{\min}\leq \frac{2dR}{L_\gamma(R-d)}\in\mathcal{O}\left(\frac d L_\gamma\right),
\end{align}
for $L_\gamma\to 0+$ and $d\to 0+$.
\ml
Now, let $\gamma$ be parametrized by arc length $s$, so that $x(\theta_0)=\gamma(0)$. W.l.o.g. assume that $r'(0)\leq 0$. (Otherwise, we consider the mirror image.) Then $\sin|\theta-\phi|_{\min}=-r'(0)$; see Figure \ref{diff}.

\begin{figure}[h]
\begin{center}
\includegraphics[width=0.45\textwidth]{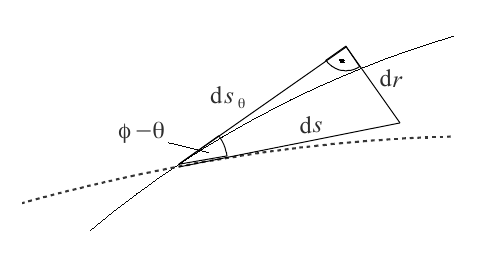}
\caption{Relationship between $ds$, $dr$ and the difference $\phi-\theta$. The dotted curve depicts $\gamma$, the continuous curve  shows the central circular curve of the ring segment $\mathcal{R}$.}\label{diff}
\end{center}
\end{figure}

The proof of Lemma \ref{lem14} shows that
\begin{align*}
\kappa(s)&=\theta'^2(s)r(s)-r''(s),\\
d(s)&=R-r(s).
\end{align*}
Note that 
\begin{align*}
r(s) & = r(0) + r'(0) s + o(s^2)\\
&  = R - s\,\sin|\theta-\phi|_{\min} + o(s^2), \quad\text{as $s\to 0+$.}
\end{align*}
Therefore, for $d\to 0+$,
\begin{align*}
d(0)&\in R-R+\mathcal{O}(d)\in\mathcal{O}(d),\\
d'(0)&=-r'(0)=\sin|\theta-\phi|_{\min}\in\mathcal{O}\left(\frac d L_\gamma\right),
\end{align*}
and
\begin{align*}
d''(0) &=\kappa(0)-\theta'^2(0)r(0)\in\kappa(0)-\frac1{R^2}(R-\mathcal{O}(d))\\
&\in\kappa(0)-\frac1R+\mathcal{O}(d).
\end{align*}
By Lemma \ref{lem14}, we can write this as
\begin{align*}
d''(0)\in\mathcal{O}\left(\frac{d}{L_\gamma^2}\right)+\mathcal{O}(d) \in\mathcal{O}\left(\frac{d}{L_\gamma^2}\right).
\end{align*}
Therefore, the Taylor expansion of $d$ in terms of $L_\gamma$ yields
\begin{align*}
d(L_\gamma)\in\mathcal{O}(d)+L_\gamma\mathcal{O}\left(\frac{d}{L_\gamma}\right)+L^2_\gamma\mathcal{O}\left(\frac d {L_\gamma^2}\right)+\mathcal{O}(L_\gamma^3),
\end{align*}
which implies the claim for $L_\gamma\to 0 +$.\qedhere

\end{proof}

\begin{theorem}\label{main}
Under the same hypotheses and with the same notation as in Lemmas \ref{lem12}, \ref{lem13}, and \ref{lem14}, we obtain for all $x\in\mathcal{R}\cap \gr\gamma$ the estimate
\[
\left|R-\frac1{\kappa(x)}\right|\in\mathcal{O}(d^{1/3}).
%\left|R-\frac1{\kappa(x)}\right|\in\mathcal{O}(d^{1-\frac2l})
\]
\end{theorem}
\begin{proof}
Substituting the result from Lemma \ref{ring2} into that of Lemma \ref{lem14} gives
\begin{align*}
&\left|R-\frac 1 {\kappa(x)}\right|\in\mathcal{O}\left(\frac d {L_\gamma^2}\right)\in\mathcal{O}\left(\frac d {d^{2/3}}\right)\in\mathcal{O}(d^{1/3}). \qedhere
\end{align*}
\end{proof}

The above theorem now implies the uniform multigrid convergence of the curvature estimator $\widehat{\kappa}_{MDCA}$.

\begin{theorem}\label{haupt}
Suppose $\mathbb{X}$ is a nonempty family of convex and compact subsets of $\mathbb{R}^2$. Further suppose that for all $X\in\mathbb{X}$ the curvature along the boundary $\partial X$ is continuously differentiable, strictly positive, and bounded from above and below by some positive constants. 

Then the multigrid curvature estimator $\widehat{\kappa}_{MDCA}$ is uniformly multigrid convergent for $\mathbb{X}$ with $\tau(h)\in\mathcal{O}(h^{1/3})$.
\end{theorem}

\begin{figure}[h]
\begin{center}
\includegraphics[width=0.4\textwidth]{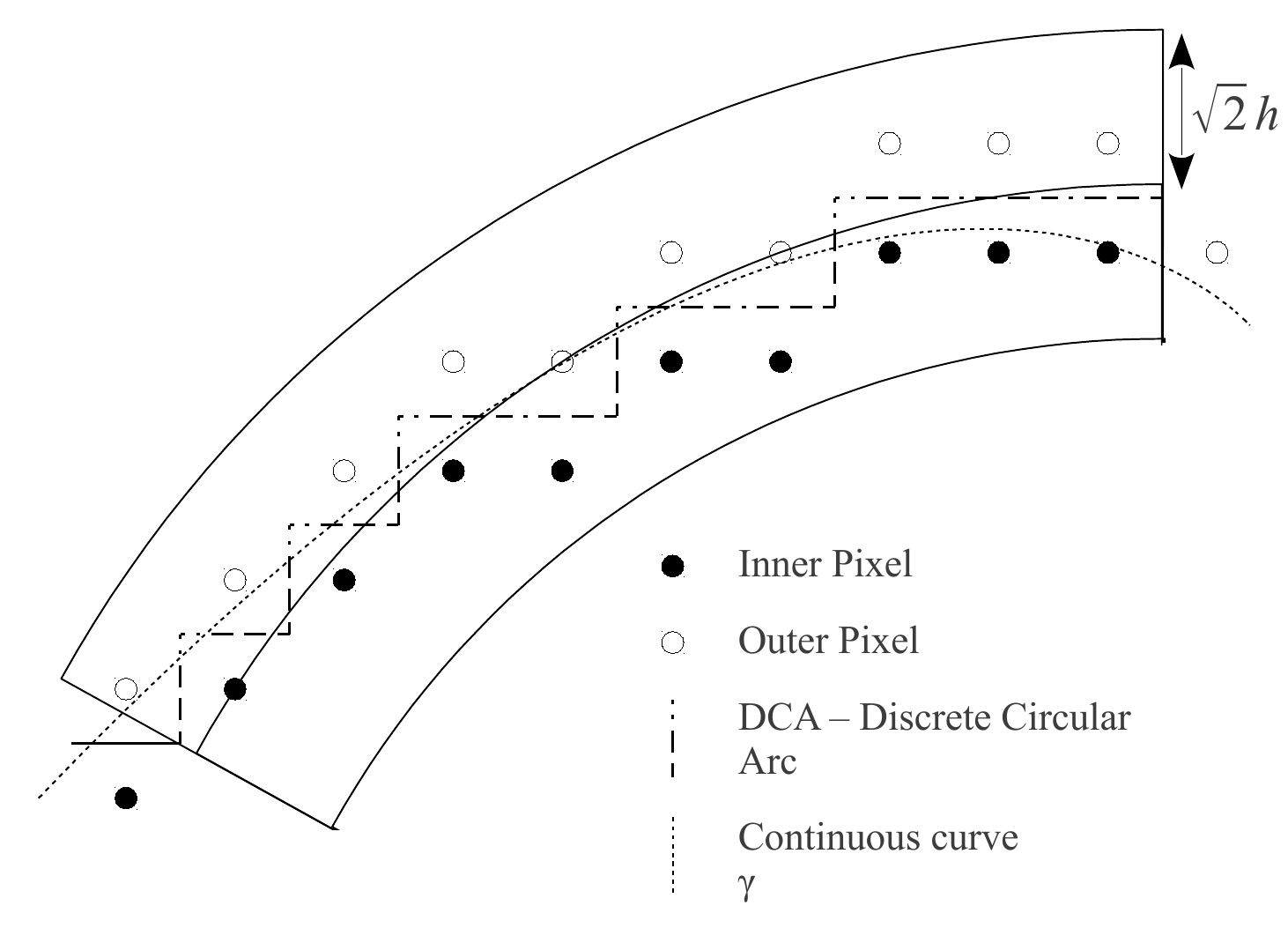}
\caption{Illustration of the proof of  Theorem \ref{haupt}}\label{proofgraph}
\end{center}
\end{figure}

\begin{proof}
The original proof can be found in \cite{rouss}. We reproduce it here in our notation and terminology since it uses arguments that will also be employed in Theorem \ref{haupt2} below, which uses weaker hypotheses.
 
Assume that $X\in\mathbb{X}$, $h>0$, $x\in\partial X$, $e\in\partial D_h(X)$ and $y\in e$ with $\|x-y\|_1<h$ are given. Suppose $(A_i)_{i\in\{1,\dots,n\}}$ is the set of MDCAs for $\partial D_h(X)$. Then $y$ is closer to $e (A_i)$ than to any other edge. Let $\mathcal{K}(m,R)$ be the circle with smallest radius that separates the inner pixels $\{p^i\}$ and outer pixels $\{p^o\}$ of $A_i$. Let $\theta_1$ and $\theta_2$  be the polar angles of the first and last edge of $A_i$ with respect to the center $m$. Let $\mathcal{R}$ be the associated $(m,R,\sqrt{2}h,\theta_1,\theta_2)$-ring. As $\mathcal{K}(m,R)$ separates the pixels $\{p^i\}$ from the pixels $\{p^o\}$ and $d$ was chosen to be equal to $\sqrt{2}h$, all these pixels lie inside of $\mathcal{R}$. Moreover, since $\partial X$ lies between $\{p^i\}$ and $\{p^o\}$ it is also contained in $\mathcal{R}$. It follows from \cite{topo} that for small enough $h>0$ the topologies of $\partial D_h(X)$ and $\partial X$ coincide and thus also those of $\mathcal{R}\cap\partial D_h(X)$ and $\mathcal{R}\cap\partial X$. Hence, $\mathcal{R}$ simply covers every curve with graph $\partial X$. In addition, $x,y\in \mathcal{R}$. All hypotheses of Theorem \ref{main} are thus satisfied and we can conclude that for all $x\in\mathcal{R}$ the following estimate applies:
\begin{align*}
&|\kappa(X,x)-\widehat{\kappa}^h(\partial D_h(X),e)|=|\kappa(X,x)-1/R|\\
&=|R-1/\kappa(x)|\cdot|\kappa(x)/R|\in
\mathcal{O}(d^{1/3}) \in \mathcal{O}(h^{1/3}).
\qedhere
\end{align*}
\end{proof}

\begin{theorem}\label{haupt2}
Suppose $\mathbb{X}$ is a nonempty family of compact subsets of $\mathbb{R}^2$. Further suppose that for all $X\in\mathbb{X}$ the curvature along the boundary $\partial X$ is bounded above and below by some positive constants. If the function $x\mapsto\kappa(X,x)$ is continuously differentiable at $x_0$, $\forall\,x_0\in\partial X$ with $\kappa(X,x_0)\neq 0$, then $\forall e\in\partial D_h(X)$ and $\forall y\in e$ with $\|x_0-y\|_1\leq h$, the following estimate holds:
\begin{align*}
|\widehat{\kappa}^h(\partial D_h(X),e)-\kappa(X,x_0)|\in\mathcal{O}(h^{1/3}).
\end{align*}
\end{theorem}
\begin{proof}
Assume that $X\in\mathbb{X}$ and $x_0\in\partial X$ are satisfying the hypotheses set forth in the theorem. Moreover, for an $h>0$,  let $e\in D_h(X)$ and $y\in e$ with $\|x_0-y\|_1<h$ be given. As in the proof of Theorem \ref{haupt} we define an $(m,R,\sqrt{2}h,\theta_1,\theta_2)$-ring $\mathcal{R}^h$ depending on $h$. As $\kappa(X,x_0)\neq 0$ and the mapping $x\mapsto\kappa(X,x)$ is continuously differentiable at $x_0$, there exists a neighborhood $U$ of $x_0$ such that $U\cap \partial X$ is the graph of a convex curve with strictly positive or negative curvature. Assume w.l.o.g. that the curvature is strictly positive. From Lemma \ref{lem14} we infer that the length of a maximal digital arc goes to zero as $h\to 0$. Therefore, there exists an $h_0>0$ so that for all $h\leq h_0$ the intersection $\mathcal{R}^h\cap \partial X$ is the graph of a convex curve with strictly positive and continuously differentiable curvature. The claim now follows from the arguments given in the proof of Theorem \ref{haupt}.
\end{proof}
\section{Experimental valuation}\label{ergeb}
In this section, we evaluate the MDCA and $\lambda$--MDCA curvature estimator. For this purpose, we use the following four segmented test objects, i.e.,  planar curves for which the exact curvature profiles can be computed explicitly. Except for the ellipse, these planar curves are not globally convex. However, they are taken as piecewise convex curves. In concave regions, we consider the inversely parametrized curve.

\begin{figure}[h!]
\centering
\subfigure[Ellipse]{\includegraphics[width=0.2\textwidth]{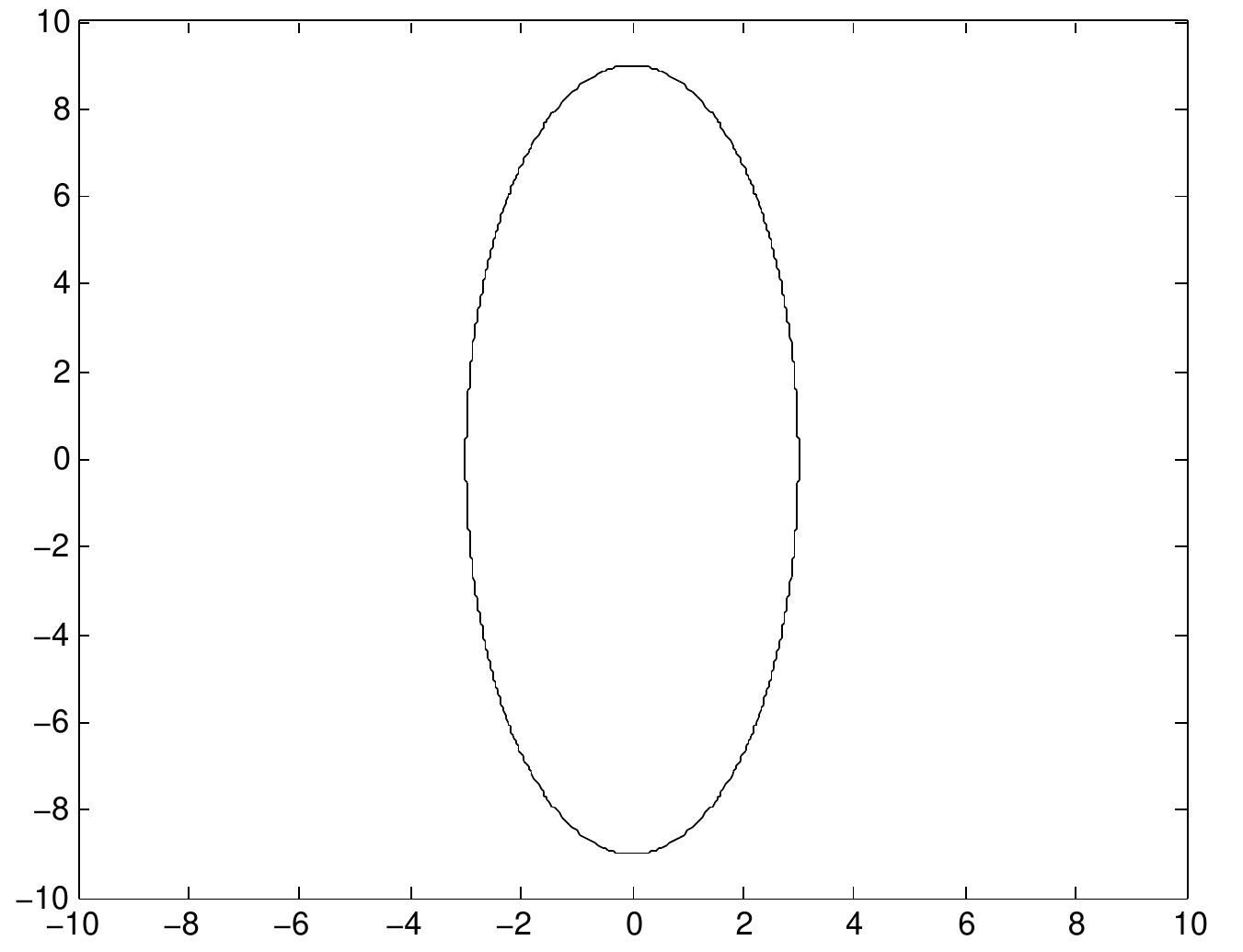}}
\subfigure[Gummy Bear]{\includegraphics[width=0.2\textwidth]{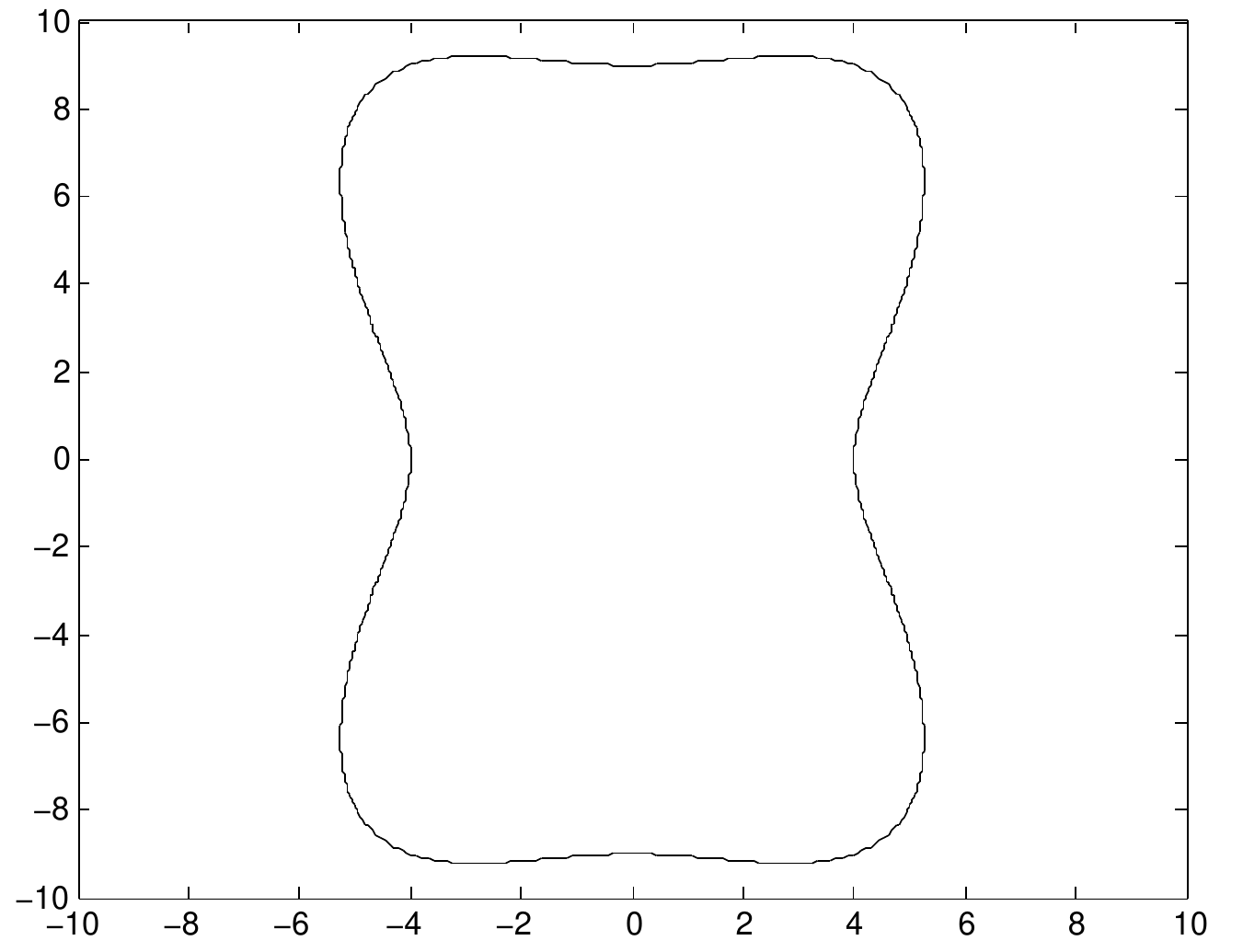}}
\subfigure[Hour Glass]{\includegraphics[width=0.2\textwidth]{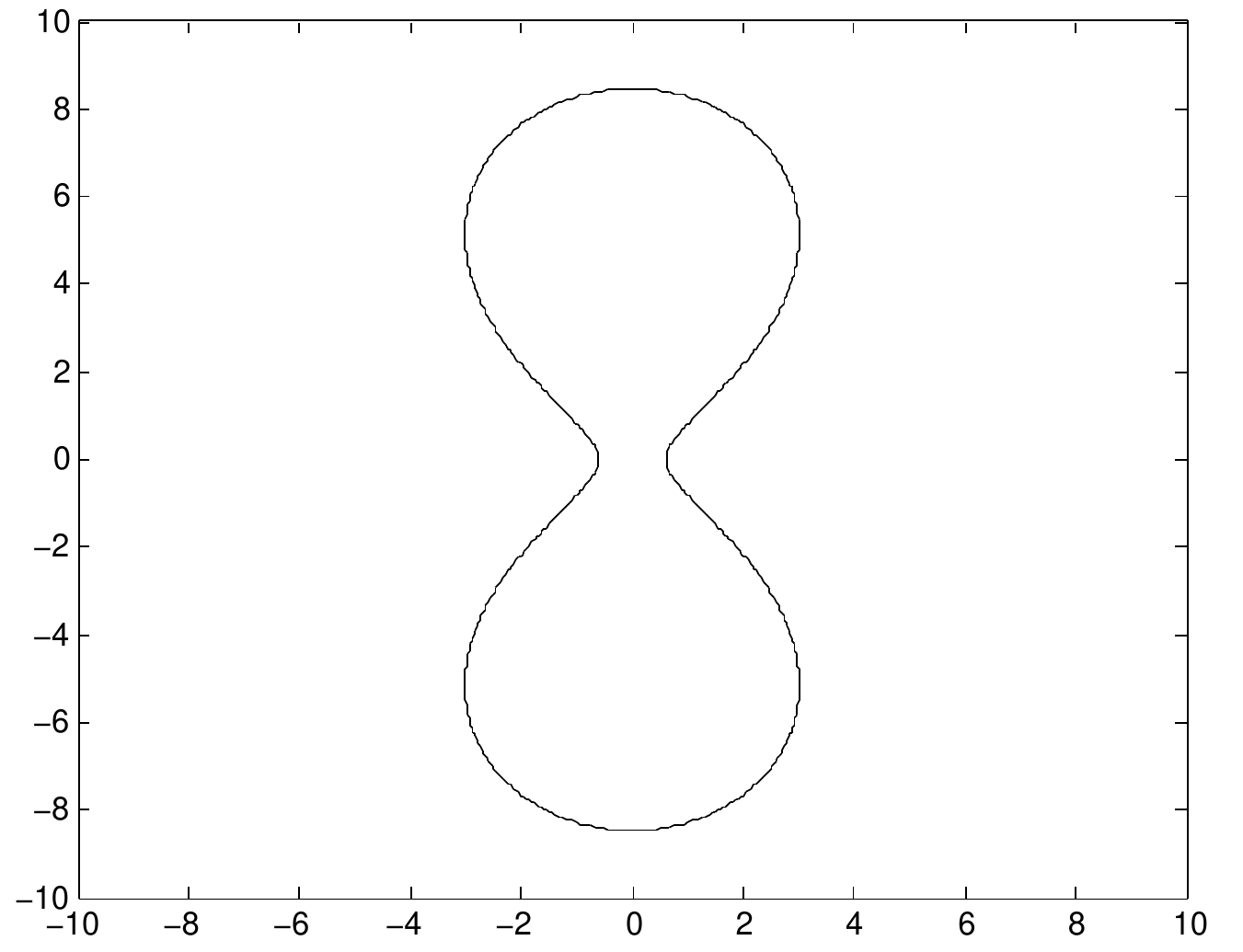}}
\subfigure[Rhombus]{\includegraphics[width=0.2\textwidth]{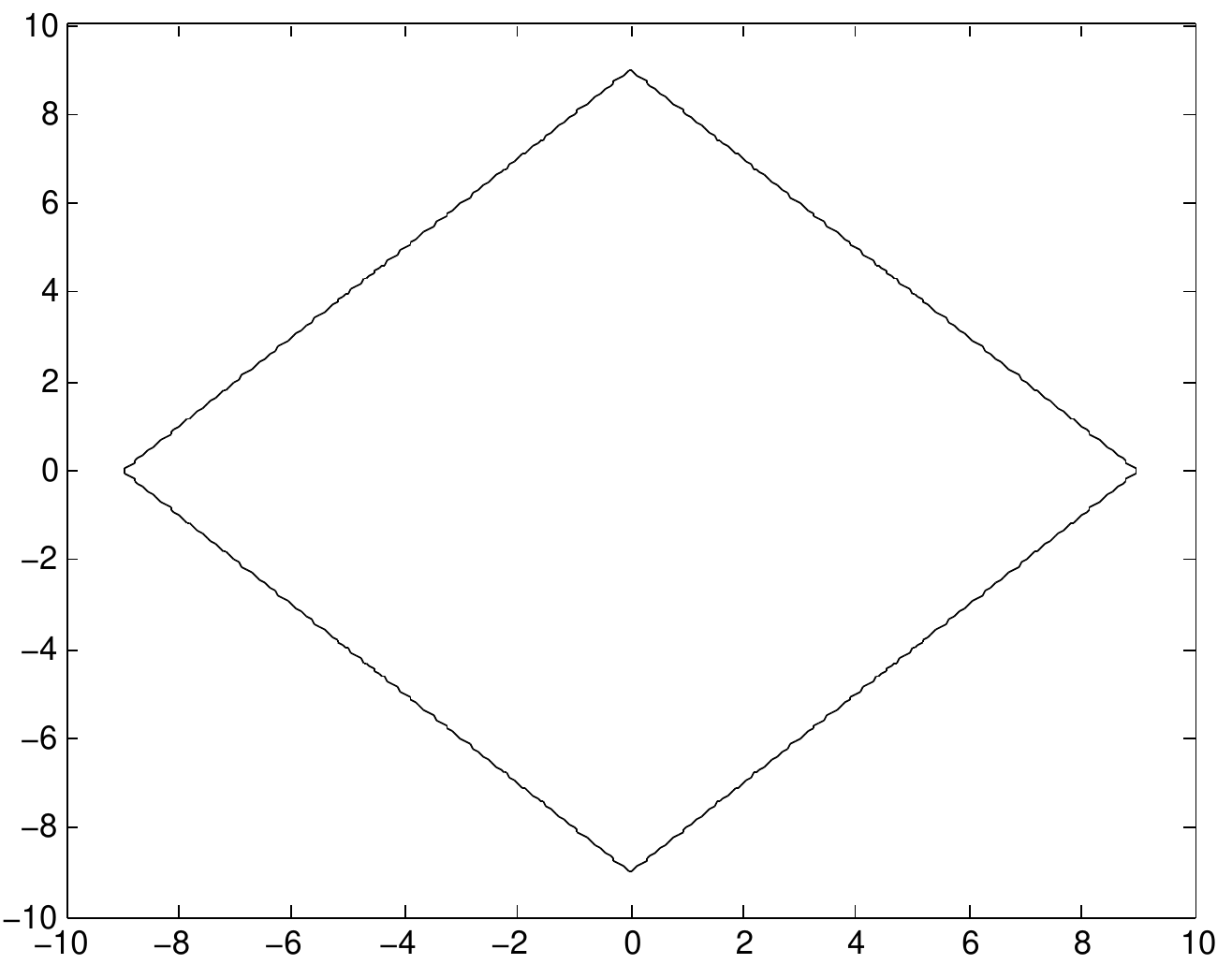}}
%\subfigure[Ellipse]{\includegraphics[width=0.2\textwidth]{Bilder/ell1.pdf}}
%\subfigure[Gummy Bear]{\includegraphics[width=0.2\textwidth]{Bilder/gummi1.pdf}}
%\subfigure[Hour Glass]{\includegraphics[width=0.2\textwidth]{Bilder/sandd1.pdf}}
%\subfigure[Rhombus]{\includegraphics[width=0.2\textwidth]{Bilder/quad.pdf}}
\caption{Test Objects (a)--(b) fulfill the requirements of Theorem \ref{haupt2}, whereas (d) does not, since it has corners.}\label{exobjects}
\end{figure}

%\begin{comment}
The first test object (a) is an ellipse 
$$
E:=\left\{\left({{x}\atop{y}}\right)\in\mathbb{R}^2\;\mid\;\frac{x^2}{a^2}+\frac{y^2}{b^2}= 1\right\}
$$
with semi axes $a=9$ and $b=3$.
The second test object, which was termed ``gummy bear,'' is represented by the point set
\begin{align*}
G:=\left\{\left({{x}\atop{y}}\right)\in\mathbb{R}^2\;\mid\;\left(\frac{x}{3}\right)^4+\left(\frac{y}{2}\right)^4-x^2-y^2= 1\right\},
\end{align*}
and the third object, named ``hour glass,'' by the point set
\begin{scriptsize}
\begin{align*}
S := & \left\{ \left({{x}\atop{y}}\right)\in\R^2\;|\; \left(3.96+\left(\frac{x}{3}\right)^2+\left(\frac{y}{3}\right)^2\right)^2
%\right.\right.
%\\
%& \left. \quad
-15.84\cdot\left(\frac{x}{3}\right)^2= 16\right\}.
\end{align*}
\end{scriptsize}
We consider the object (d), the rhombus, in more detail in Section \ref{limits}.

The exact curvature profile along the curve $\gamma_X$ with graph $\gr \gamma_X = X$ for the each test object  $X\in\{E,G,S\}$ can be explicitly computed by means of the following formula (see, e.g.  \cite{docarmo,kuehnel,montiel}):
\begin{align*}
\kappa=\frac{g_{xx}g_y-2g_{xy}g_xg_y+g_{yy}g_x}{(g_x^2+g_y^2)^{3/2}},
%=\frac{4x+4y}{a^2b^2\left((2x/a^2)^2+(2y/b^2)^2\right)^{\frac32}}.
\end{align*}
where $X$ satisfies the implicit equation $g(x,y) = 0$.

%\end{comment}
%---------------------------------------------------------------------------
\subsection{Experimental Error}
The four test objects are Gau\ss--discretized (see Section \ref{sec imaging}) using $\,h=2^{-n}$ with $n\in\{0,\dots,6\}$. Next, the MDCA and $\lambda$--MDCA curvature estimators are computed and compared to the exact curvature calculation. In order to compare the error at pixel edge $e$ with the exact value, we consider the point $p'=(x',y')^{T}\in \partial X$ whose distance to the midpoint of $e$ is minimal. The absolute error at $e$ for resolution $h$ is then given by
\begin{align*}
\epsilon_{\mathrm{abs}}^h(e):=|\kappa(X,p')-\widehat{\kappa}_{\mathrm{MDCA}}^h(\partial D_h(X),e)|.
\end{align*}
We define the average error by
\begin{align*}
\epsilon_{\mathrm{av}}^h:= \frac1{|D_h(X)|}\sum_{e\in D_h(X)}\epsilon_{\mathrm{abs}}^h(e).
\end{align*}
Here, $|D_h(X)|$ denotes the number of edges of the simple closed digital curve whose graph is $D_h(X)$. Furthermore, we define the maximal error as
\begin{align*}
\epsilon_{\mathrm{\max}}^h:= \max_{e\in D_h(X)}\epsilon_{\mathrm{abs}}^h(e).
\end{align*}

In Figures \ref{errorgraph1} and \ref{errorgraph2} the average and maximal errors are plotted against the resolution $h$. The plots are double logarithmically scaled. Using this particular scaling of the axes, we see that the slope of the regression line for the error values equals the experimental  speed of convergence:
\begin{align*}
\log (\epsilon)=m\cdot \log (h)+ {C}&\Longleftrightarrow \log (\epsilon)=\log (h^m)+{C}\\& \Longleftrightarrow \epsilon={C}\cdot h^m,
\end{align*}
where $C\in \R$.
\begin{figure}[h!]
\begin{center}
\includegraphics[width=0.45\textwidth]{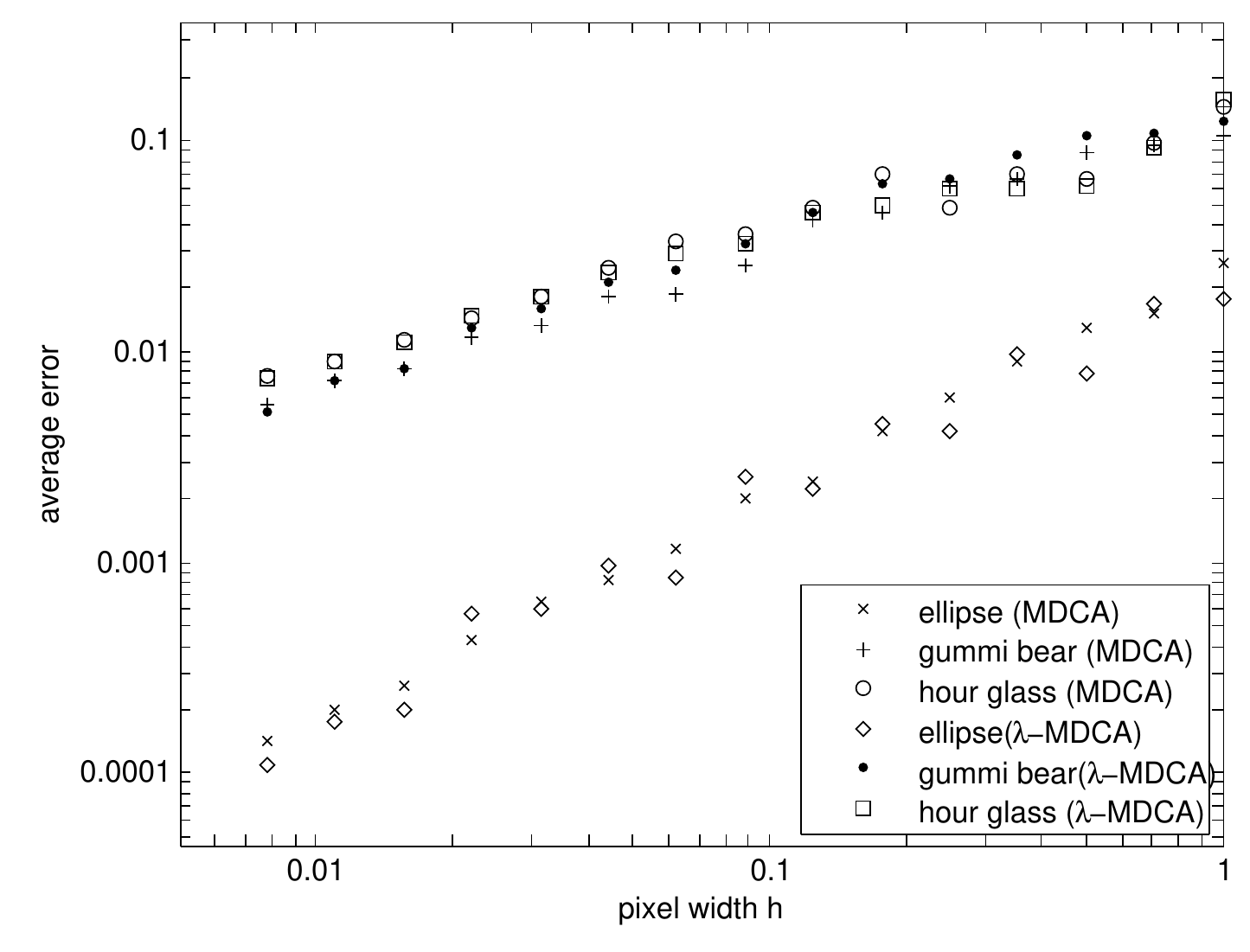}
\caption{Average Error for the MDCA and the $\lambda$-MDCA curvature estimators evaluated at the smooth test objects for varying pixel width $h$.}\label{errorgraph1}
\end{center}
\end{figure}
\begin{figure}[h!]
\begin{center}
\includegraphics[width=0.45\textwidth]{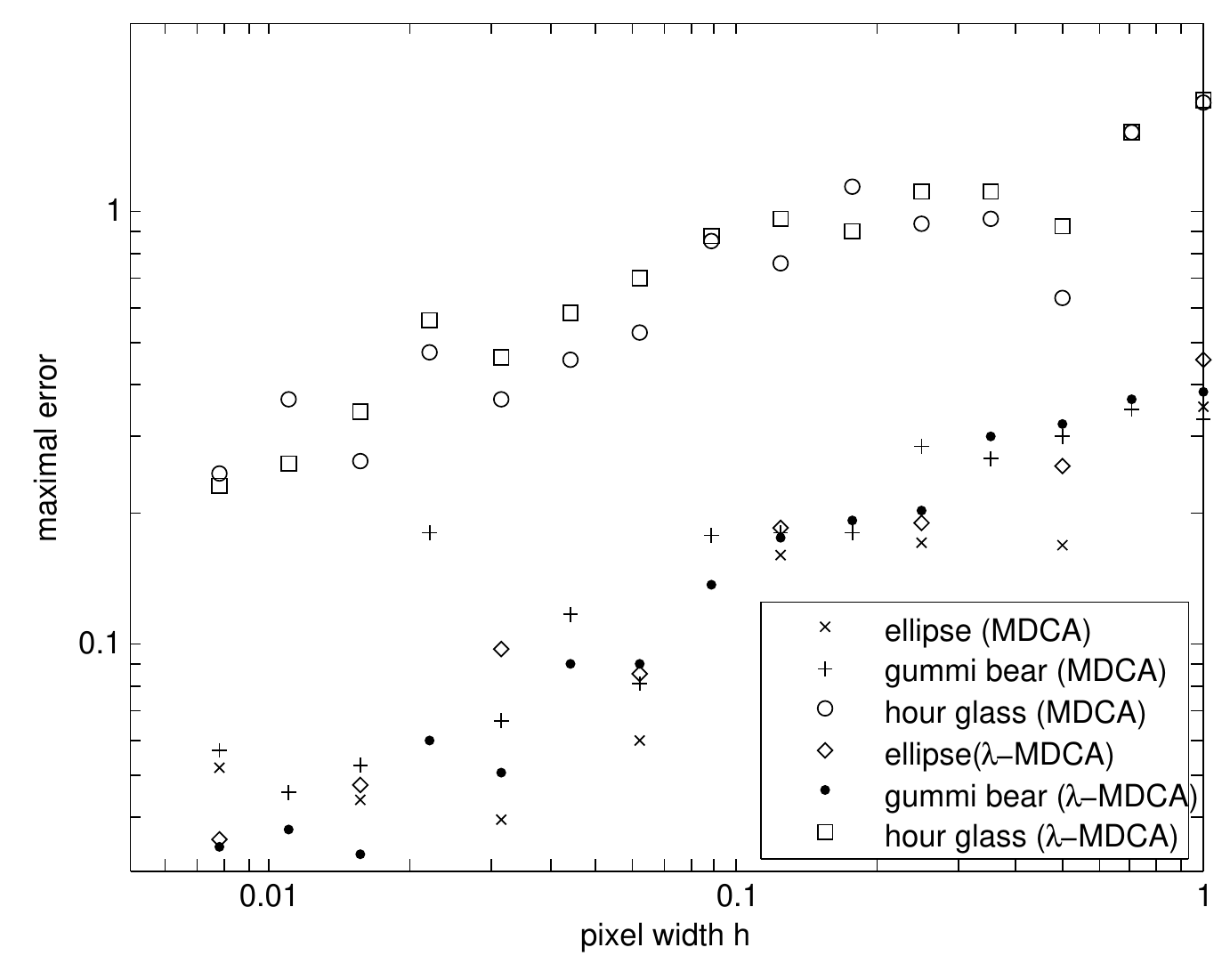}
\caption{Maximal Error for the MDCA and the $\lambda$-MDCA curvature estimators evaluated at the smooth test objects for varying pixel width $h$. The $\lambda$-MDCA curvature estimator shows a higher order of convergence than the standard MDCA-estimator; compare also with Table \ref{errtab}.}\label{errorgraph2}
\end{center}
\end{figure}

The plots indicate that for the average error there is hardly any difference between the $\lambda$--MDCA and the MDCA curvature estimator. However, when the maximum error is considered the $\lambda$--MDCA curvature estimator exhibits a faster speed of convergence. For the calculations, the weight function $\lambda(t)=-t\log(t)-(1-t)\log(1-t)$ was used.

The experimental convergence speeds for the three test objects are summarized in Table \ref{errtab}. One observes the following: Firstly, the convergence of the maximal error for the test object ``hour glass'' is the slowest with $\mathcal{O}(h^{0.36})$ and this value is already close to the theoretical limit of $\mathcal{O}(h^{\frac13})$. Secondly, the convergence of the average error for the test object ``ellipse'' is even faster than $\mathcal{O}(h)$.
\begin{table}[h!] 
\centering
\renewcommand{\arraystretch}{1.5}
\begin{tabular}{|l|ccc|}
\hline
&Ellipse&Gummy Bear &Hour Glass\\ \hline
Avg. Error (MDCA)&$\mathcal{O}(h^{1.08})$&$\mathcal{O}(h^{0.64})$&$\mathcal{O}(h^{0.57})$\\ 
Max. Error (MDCA)&$\mathcal{O}(h^{0.44})$&$\mathcal{O}(h^{0.42})$&$\mathcal{O}(h^{0.36})$\\ 
Avg. Error ($\lambda$-MDCA)&$\mathcal{O}(h^{1.06})$&$\mathcal{O}(h^{0.68})$&$\mathcal{O}(h^{0.56})$\\ 
Max. Error ($\lambda$-MDCA)&$\mathcal{O}(h^{0.50})$&$\mathcal{O}(h^{0.56})$&$\mathcal{O}(h^{0.37})$\\ \hline
\end{tabular}
\smallskip
\caption{Experimental Speed of Convergence}\label{errtab}
\end{table}

%---------------------------------------------------------------------------
\subsection{Limits of the Method}\label{limits}
A test object for which the MDCA curvature estimator performs very badly is the rhombus (see Figure \ref{exobjects} (d)). It is represented by the point set
\[
Q:=\{(x,y)^T\in\R^2\;|\; |x|+|y|\leq 9\}.
\]
Every boundary point of the set $Q$ was excluded in Theorem \ref{haupt2}: At the four corners of the rhombus, the curvature is not defined, i.e., unbounded, and for the remaining boundary points the curvature is equal to zero. Hence, it is not possible to employ Theorem \ref{haupt2} to prove convergence to the actual curvature $\kappa(Q,x)$. This is illustrated in Figure \ref{corners} which shows that for a resolution of $h = 2^{-2}$ the curvature is badly approximated in a neighborhood of the corners. 

The reason for this discrepancy lies in the fact that the entire side of the rhombus is one MDCA and that the midpoint of the subsequent MDCA is located next to the corner. Hence, all points on the side of the rhombus that are closer to a corner than to the midpoint of the side use the MDCA at the corner. For this MDCA, however, the curvature is substantially larger than zero. From the mathematical point of view, the curvature of the rhombus is not bounded below by a positive constant, thus the requirements of Theorem \ref{haupt2} are not met.

The $\lambda$--MDCA curvature estimator circumvents this problem by assigning higher weights to the digital arcs when a boundary point moves to a more central position. As an MDCA at the side of the rhombus is significantly longer than one at a corner, it is weighted higher and therefore a better result is obtained; see Figure \ref{corners2}.

\begin{figure}[h]
\centering
\subfigure[Graphical Representation]{\includegraphics[width=0.24\textwidth]{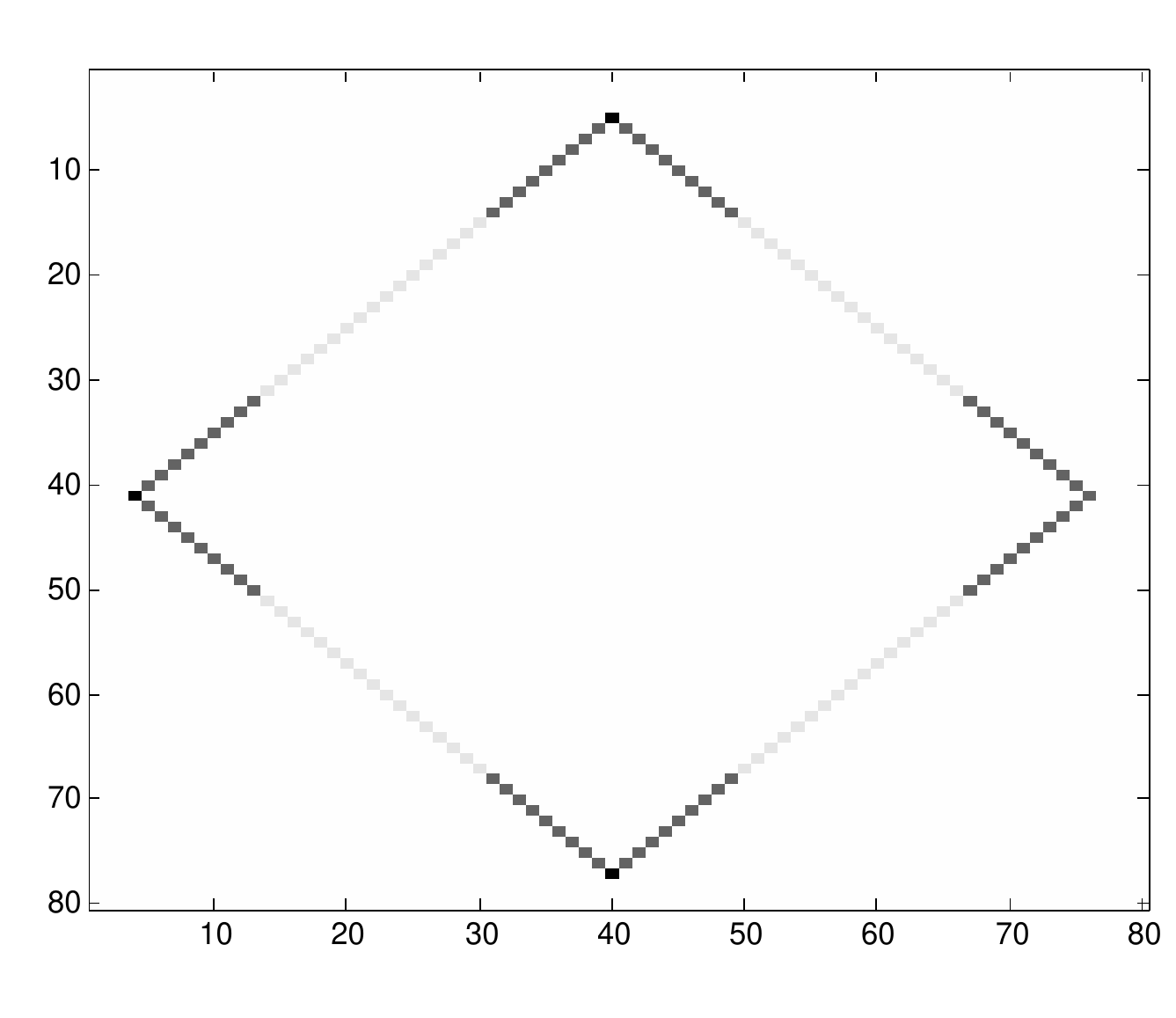}}
\subfigure[Curvature Profile]{\includegraphics[width=0.24\textwidth]{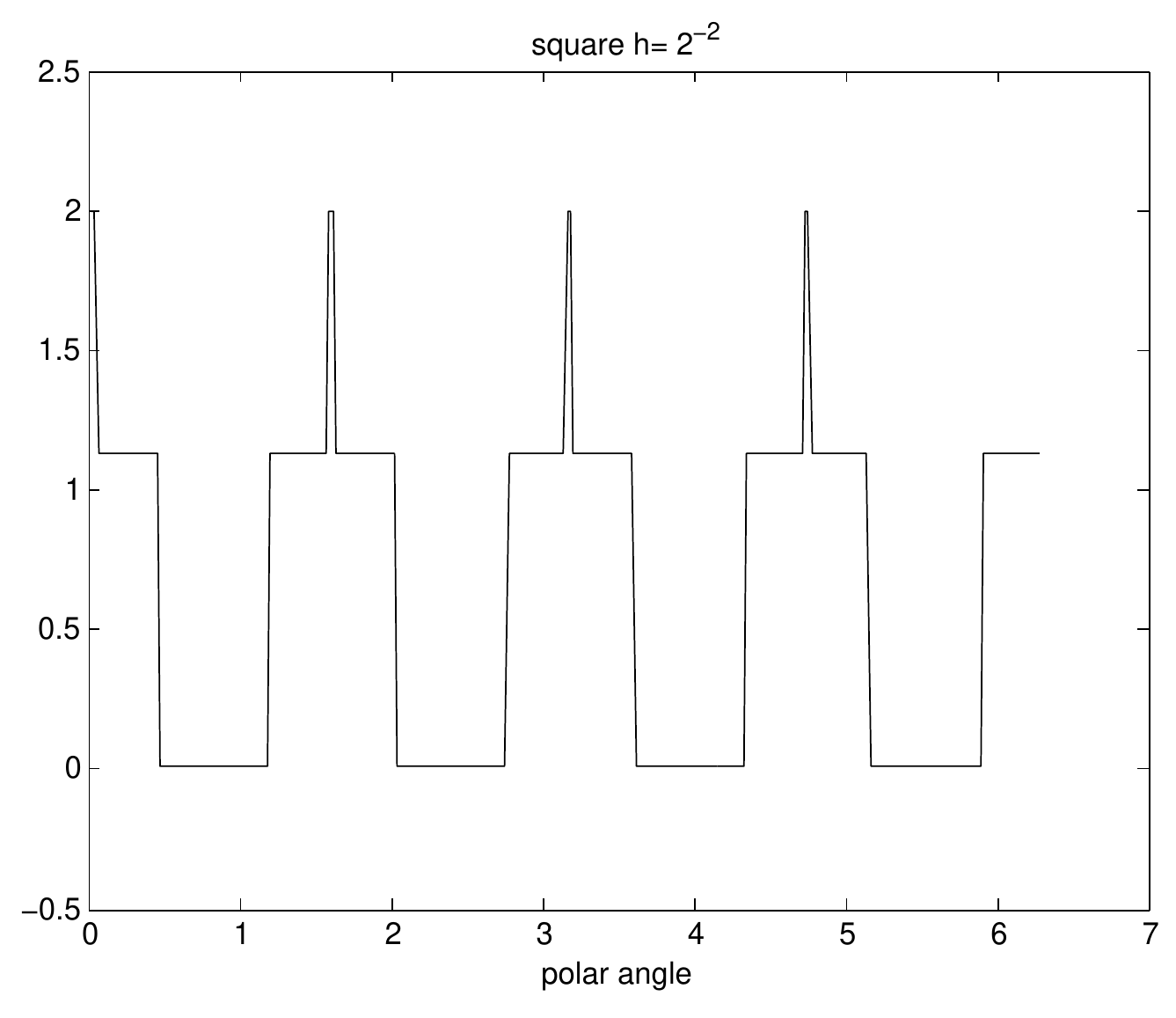}}
%\subfigure[Graphical Representation]{\includegraphics[width=0.24\textwidth]{Bilder/quadim.pdf}}
%\subfigure[Curvature Profile]{\includegraphics[width=0.24\textwidth]{Bilder/quadprof.pdf}}
\caption{MDCA Curvature Estimator for the Rhombus}\label{corners}
\end{figure}

\begin{figure}[h]
\centering
\subfigure[Graphical Representation]{\includegraphics[width=0.24\textwidth]{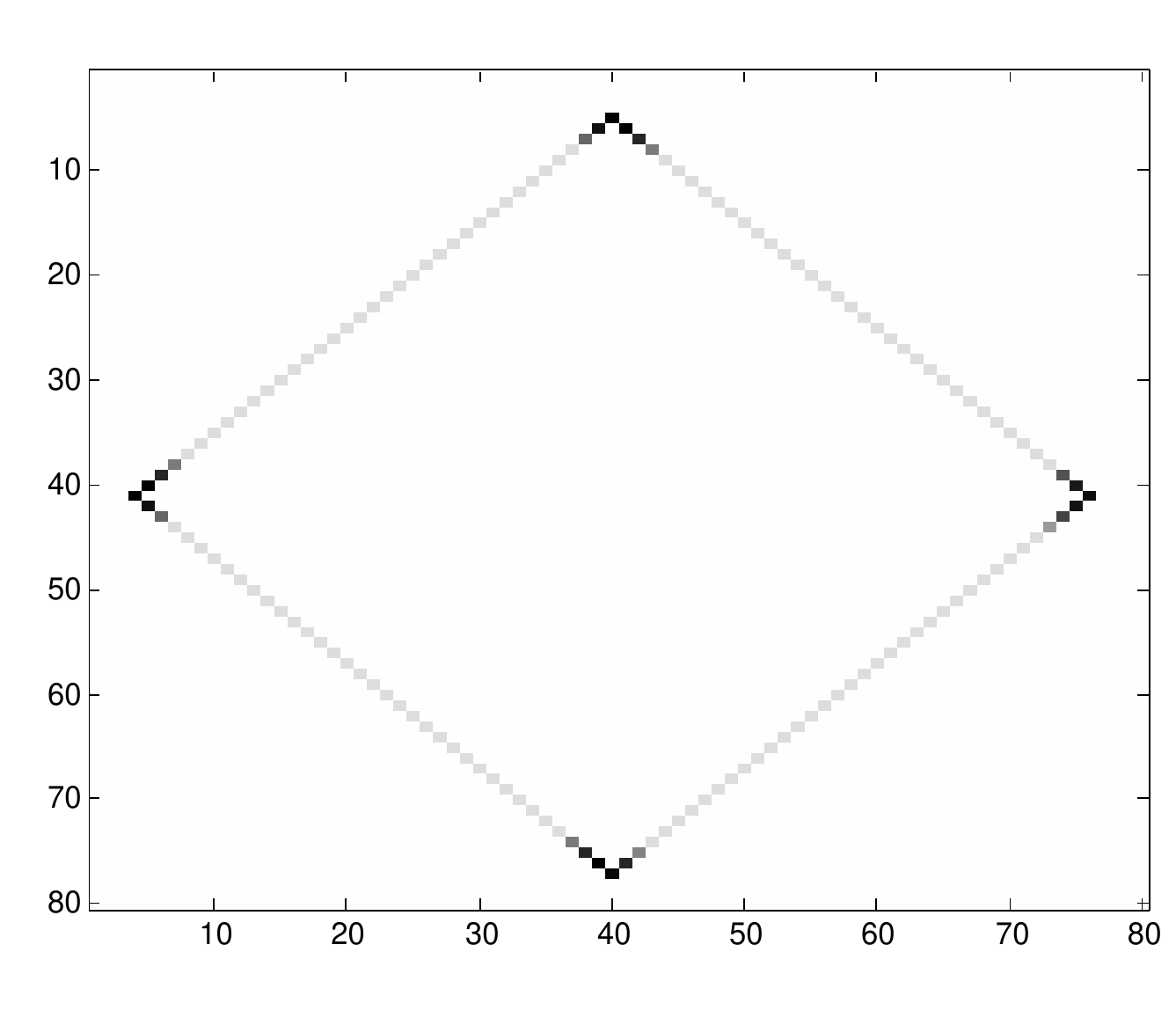}}
\subfigure[Curvature Profile]{\includegraphics[width=0.24\textwidth]{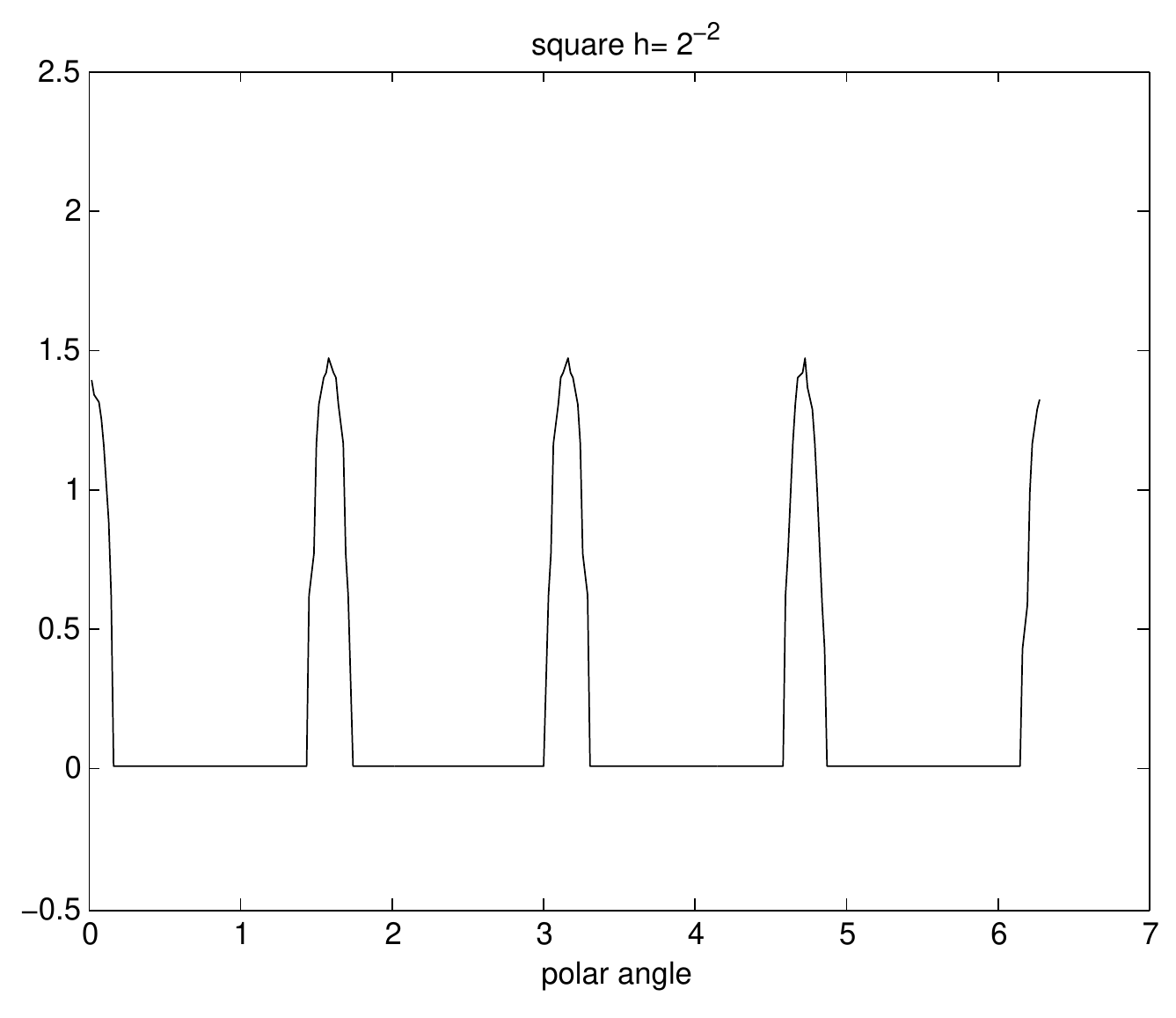}}
%\subfigure[Graphical Representation]{\includegraphics[width=0.24\textwidth]{Bilder/quadimlam.pdf}}
%\subfigure[Curvature Profile]{\includegraphics[width=0.24\textwidth]{Bilder/quadproflam.pdf}}
\caption{$\lambda$--MDCA Curvature Estimator for the Rhombus}\label{corners2}
\end{figure}
%
%Chapter8:
\section{Conclusion}
\label{sec conclusions}
We analyzed the parameter-free MDCA curvature estimator presented in \cite{rouss}, which is based on maximal digital circular arcs (MDCA). In addition, we extended this estimator to the new $\lambda$-MDCA curvature estimator, which is smoother and does also not require any parameter. Especially around corners this estimator provides better results then the MDCA estimator.

For convex sets with positive, continuous curvature profile we provided a proof for its uniform multigrid convergence as defined in \cite{klette} with convergence speed of $\mathcal{O}(h^{\frac13})$. 
In an experimantal evaluation we observed that this speed of convergence is almost reached for some sets and therefore 
seems to be the optimal upper bound.

The respective results on the optimal multigrid convergence rate for $\lambda$-MDCA curvature estimator are still an open question.

%
%%
%\appendices
%\section{Proof of the First Zonklar Equation}
%Appendix one text goes here.
%
% you can choose not to have a title for an appendix
% if you want by leaving the argument blank
%\section{}
%Appendix two text goes here.
%
%
% use section* for acknowledgement
%\section*{Acknowledgment}
%
%
%The authors would like to thank...

% Can use something like this to put references on a page
% by themselves when using endfloat and the captionsoff option.
\ifCLASSOPTIONcaptionsoff
  \newpage
\fi

% trigger a \newpage just before the given reference
% number - used to balance the columns on the last page
% adjust value as needed - may need to be readjusted if
% the document is modified later
%\IEEEtriggeratref{8}
% The "triggered" command can be changed if desired:
%\IEEEtriggercmd{\enlargethispage{-5in}}

% references section

% can use a bibliography generated by BibTeX as a .bbl file
% BibTeX documentation can be easily obtained at:
% http://www.ctan.org/tex-archive/biblio/bibtex/contrib/doc/
% The IEEEtran BibTeX style support page is at:
% http://www.michaelshell.org/tex/ieeetran/bibtex/
%\bibliographystyle{IEEEtran}
% argument is your BibTeX string definitions and bibliography database(s)
%\bibliography{IEEEabrv,../bib/paper}
%
% <OR> manually copy in the resultant .bbl file
% set second argument of \begin to the number of references
% (used to reserve space for the reference number labels box)

\bibliographystyle{IEEEtran}
\bibliography{curvature}

% Generated by IEEEtran.bst, version: 1.13 (2008/09/30)
\begin{thebibliography}{10}
\providecommand{\url}[1]{#1}
\csname url@samestyle\endcsname
\providecommand{\newblock}{\relax}
\providecommand{\bibinfo}[2]{#2}
\providecommand{\BIBentrySTDinterwordspacing}{\spaceskip=0pt\relax}
\providecommand{\BIBentryALTinterwordstretchfactor}{4}
\providecommand{\BIBentryALTinterwordspacing}{\spaceskip=\fontdimen2\font plus
\BIBentryALTinterwordstretchfactor\fontdimen3\font minus
  \fontdimen4\font\relax}
\providecommand{\BIBforeignlanguage}[2]{{%
\expandafter\ifx\csname l@#1\endcsname\relax
\typeout{** WARNING: IEEEtran.bst: No hyphenation pattern has been}%
\typeout{** loaded for the language `#1'. Using the pattern for}%
\typeout{** the default language instead.}%
\else
\language=\csname l@#1\endcsname
\fi
#2}}
\providecommand{\BIBdecl}{\relax}
\BIBdecl

\bibitem{nbe}
I.~T. Young, J.~E. Walker, and J.~E. Bowie, ``An analysis technique for
  biological shape. i,'' \emph{Information and Control}, vol.~25, no.~4, pp.
  357--370, 1974.

\bibitem{canham}
P.~B. Canham, ``The minimum energy of bending as a possible explanation of the
  biconcave shape of the human red blood cell,'' \emph{Journal of Theoretical
  Biology}, vol.~26, no.~1, pp. 61--76, January 1970.

\bibitem{nbe2}
L.~J. {van Vliet} and P.~W. Verbeek, ``Curvature and bending energy in
  digitized 2d and 3d images,'' in \emph{Proceedings of the 8th Scandinavian
  Conference on Image Analysis (SCIA)}, 1993, pp. 1403--1410.

\bibitem{heart}
J.~S. Duncan, F.~A. Lee, A.~W.~M. Smeulders, and B.~L. Zaret, ``A bending
  energy model for measurement of cardiac shape deformity,'' \emph{IEEE
  Transactions on Medical Imaging}, vol.~10, no.~3, pp. 307--320, 1991.

\bibitem{cell}
A.~Pasqualato, A.~Palombo, A.~Cucina, M.~Mariggi\`o, L.~Galli, D.~Passaro,
  S.~Dinicola, S.~Proietti, F.~D'Anselmi, P.~Coluccia, and M.~Bizzarri,
  ``Quantitative shape analysis of chemoresistant colon cancer cells:
  Correlation between morphotype and phenotype,'' \emph{Experimental Cell
  Research}, vol. 318, no.~7, pp. 835--846, 2012.

\bibitem{cell2}
R.~{Marcondes Cesar, Jr.} and L.~{da Fontoura Costa}, ``Application and
  assessment of multiscale bending energy for morphometric characterization of
  neural cells,'' \emph{Rev. Sci. Instrum.}, vol.~68, no.~5, pp. 2177--2186,
  May 1997.

\bibitem{backhaus}
A.~Backhaus, A.~Kuwabara, M.~Bauch, N.~Monk, G.~Sanguinetti, and A.~Fleming,
  ``Leafprocessor: a new leaf phenotyping tool using contour bending energy and
  shape cluster analysis.'' \emph{New Phytologist}, vol. 187, pp. 251--261,
  2010.

\bibitem{sodt}
A.~J. Sodt and R.~W. Pastor, ``Bending free energy from simulation:
  Correspondence of planar and inverse hexagonal lipid phases,'' \emph{Biophys
  J.}, vol. 104, no.~10, pp. 2202--2211, May 2013.

\bibitem{face}
G.~G. Gordon, ``Face recognition based on depth and curvature features,''
  \emph{IEEE Computer Society Conference on Computer Vision and Pattern
  Recognition}, pp. 808--810, 1992.

\bibitem{kiesel}
A.~C. Jalba, M.~H.~F. Wilkinson, J.~B. T.~M. Roerdink, M.~M. Bayer, and
  S.~Juggins, ``Automatic diatom identification using contour analysis by
  morphological curvature scale spaces,'' \emph{Machine Vision and
  Applications}, vol.~16, no.~4, pp. 217--22, 2005.

\bibitem{rouss}
T.~Roussillon and J.-O. Lachaud, ``Accurate curvature estimation along digital
  contours with maximal digital circular arcs,'' \emph{Combinatorial Image
  Analysis}, vol. 6636, pp. 43--55, 2011.

\bibitem{multi}
D.~Coeurjolly, J.~Lachaud, and T.~Roussillon, ``Multigrid convergence of
  discrete geometric estimators,'' \emph{Digital Geometry Algorithms}, vol.~2,
  pp. 395--424, 2012.

\bibitem{docarmo}
M.~{Do Carmo}, \emph{Differential Geometry of Curves and Surfaces}.\hskip 1em
  plus 0.5em minus 0.4em\relax Pearson, 1976.

\bibitem{kuehnel}
W.~K{\"u}hnel, \emph{Differentialgeometrie}, 5th~ed.\hskip 1em plus 0.5em minus
  0.4em\relax Vieweg+Teubner, 2010.

\bibitem{montiel}
S.~Montiel and A.~Ros, \emph{Curves and Surfaces}, 2nd~ed.\hskip 1em plus 0.5em
  minus 0.4em\relax AMS, 2009.

\bibitem{klettebuch}
R.~Klette and A.~Rosenfeld, \emph{Digital geometry---Geometric Methods for
  Digital Picture Analysis}.\hskip 1em plus 0.5em minus 0.4em\relax Morgan
  Kaufmann Publishers, 2004.

\bibitem{rouss_three_versions}
T.~Roussillon, L.~Tougne, and I.~Sivignon, ``On three constrained versions of
  the digital circular arc recognition problem,'' in \emph{DGCI 2009}, ser.
  LNCS 5810, S.~Brlek, C.~Rautenauer, and X.~Proven{\c{c}}al, Eds.\hskip 1em
  plus 0.5em minus 0.4em\relax Berlin Heidelberg: Springer-Verlag, 2009, pp.
  34--45.

\bibitem{lac}
J.~Lachaud, A.~Vialard, and F.~de~Vieilleville, ``Fast, accurate and convergent
  tangent estimation on digital contours,'' \emph{Image and Vision Computing},
  vol.~25, no.~10, pp. 1572--1587, 2007.

\bibitem{suppfun}
H.~Flanders, ``A proof of {M}inkowski's inequality for convex curves,''
  \emph{The American Mathematical Monthly}, vol.~75, pp. 581--593, 1968.

\bibitem{topo}
A.~Gross and L.~Latecki, ``Digitizations preserving topological and
  differential geometric properties,'' \emph{Computer Vision and Image
  Understanding}, vol.~62, no.~3, pp. 370--381, 1995.

\bibitem{klette}
R.~Klette, ``Multigrid convergence of geometric features,'' in \emph{Digital
  and Image Geometry}, ser. Lecture Notes in Computer Science, G.~Bertrand,
  A.~Imiya, and R.~Klette, Eds.\hskip 1em plus 0.5em minus 0.4em\relax
  Springer, 2001, vol. 2243, pp. 318--338.

\end{thebibliography}

\vfill

% Can be used to pull up biographies so that the bottom of the last one
% is flush with the other column.
%\enlargethispage{-5in}

% that's all folks
\end{document}